\documentclass[12pt]{article}

\usepackage{amsmath,amssymb,amsbsy,amsfonts,amsthm,latexsym,
          amsopn,amstext,amsxtra,euscript,amscd,stmaryrd,mathrsfs}

\begin{document}

\newtheorem{theorem}{Theorem}
\newtheorem{lemma}{Lemma}
\newtheorem{corollary}{Corollary}
\newtheorem{proposition}{Proposition}

\theoremstyle{definition}
\newtheorem*{definition}{Definition}
\newtheorem*{remark}{Remark}
\newtheorem*{example}{Example}

\numberwithin{equation}{section}


\def\cA{\mathcal A}
\def\cB{\mathcal B}
\def\cC{\mathcal C}
\def\cD{\mathcal D}
\def\cE{\mathcal E}
\def\cF{\mathcal F}
\def\cG{\mathcal G}
\def\cH{\mathcal H}
\def\cI{\mathcal I}
\def\cJ{\mathcal J}
\def\cK{\mathcal K}
\def\cL{\mathcal L}
\def\cM{\mathcal M}
\def\cN{\mathcal N}
\def\cO{\mathcal O}
\def\cP{\mathcal P}
\def\cQ{\mathcal Q}
\def\cR{\mathcal R}
\def\cS{\mathcal S}
\def\cU{\mathcal U}
\def\cT{\mathcal T}
\def\cV{\mathcal V}
\def\cW{\mathcal W}
\def\cX{\mathcal X}
\def\cY{\mathcal Y}
\def\cZ{\mathcal Z}


\def\sA{\mathscr A}
\def\sB{\mathscr B}
\def\sC{\mathscr C}
\def\sD{\mathscr D}
\def\sE{\mathscr E}
\def\sF{\mathscr F}
\def\sG{\mathscr G}
\def\sH{\mathscr H}
\def\sI{\mathscr I}
\def\sJ{\mathscr J}
\def\sK{\mathscr K}
\def\sL{\mathscr L}
\def\sM{\mathscr M}
\def\sN{\mathscr N}
\def\sO{\mathscr O}
\def\sP{\mathscr P}
\def\sQ{\mathscr Q}
\def\sR{\mathscr R}
\def\sS{\mathscr S}
\def\sU{\mathscr U}
\def\sT{\mathscr T}
\def\sV{\mathscr V}
\def\sW{\mathscr W}
\def\sX{\mathscr X}
\def\sY{\mathscr Y}
\def\sZ{\mathscr Z}


\def\fA{\mathfrak A}
\def\fB{\mathfrak B}
\def\fC{\mathfrak C}
\def\fD{\mathfrak D}
\def\fE{\mathfrak E}
\def\fF{\mathfrak F}
\def\fG{\mathfrak G}
\def\fH{\mathfrak H}
\def\fI{\mathfrak I}
\def\fJ{\mathfrak J}
\def\fK{\mathfrak K}
\def\fL{\mathfrak L}
\def\fM{\mathfrak M}
\def\fN{\mathfrak N}
\def\fO{\mathfrak O}
\def\fP{\mathfrak P}
\def\fQ{\mathfrak Q}
\def\fR{\mathfrak R}
\def\fS{\mathfrak S}
\def\fU{\mathfrak U}
\def\fT{\mathfrak T}
\def\fV{\mathfrak V}
\def\fW{\mathfrak W}
\def\fX{\mathfrak X}
\def\fY{\mathfrak Y}
\def\fZ{\mathfrak Z}


\def\C{{\mathbb C}}
\def\F{{\mathbb F}}
\def\K{{\mathbb K}}
\def\L{{\mathbb L}}
\def\N{{\mathbb N}}
\def\Q{{\mathbb Q}}
\def\R{{\mathbb R}}
\def\Z{{\mathbb Z}}


\def\eps{\varepsilon}
\def\mand{\qquad\mbox{and}\qquad}
\def\\{\cr}
\def\({\left(}
\def\){\right)}
\def\[{\left[}
\def\]{\right]}
\def\<{\langle}
\def\>{\rangle}
\def\fl#1{\left\lfloor#1\right\rfloor}
\def\rf#1{\left\lceil#1\right\rceil}
\def\le{\leqslant}
\def\ge{\geqslant}
\def\ds{\displaystyle}

\def\xxx{\vskip5pt\hrule\vskip5pt}
\def\yyy{\vskip5pt\hrule\vskip2pt\hrule\vskip5pt}
\def\imhere{ \xxx\centerline{\sc I'm here}\xxx }

\newcommand{\comm}[1]{\marginpar{
\vskip-\baselineskip \raggedright\footnotesize
\itshape\hrule\smallskip#1\par\smallskip\hrule}}


\def\e{\mathbf{e}}


\title{\bf Piatetski-Shapiro sequences}

\author{
{\sc Roger C.~Baker} \\
{Department of Mathematics, Brigham Young University} \\
{Provo, UT 84602 USA} \\
{\tt baker@math.byu.edu} 
\and
{\sc William D.~Banks} \\
{Department of Mathematics, University of Missouri} \\
{Columbia, MO 65211 USA} \\
{\tt bankswd@missouri.edu} 
\and
{\sc J\"org Br\"udern} \\
{Mathematisches Institut, Georg-August Universit\"at G\"ottingen} \\
{37073 G\"ottingen, Germany} \\
{\tt bruedern@uni-math.gwdg.de}
\and
{\sc Igor E.~Shparlinski} \\
{Department of Computing, Macquarie University} \\
{Sydney, NSW 2109, Australia} \\
{\tt igor.shparlinski@mq.edu.au}
\and
{\sc Andreas J.~Weingartner} \\
{Department of Mathematics, Southern Utah University} \\
{Cedar City, UT 84720 USA} \\
{\tt weingartner@suu.edu}}

\date{\today}
\pagenumbering{arabic}

\maketitle

\newpage

\begin{abstract}
We consider various arithmetic questions for the
Piatetski-Shapiro sequences $\fl{n^c}$ ($n=1,2,3,\ldots$)
with $c>1$, $c\not\in\N$.  We exhibit a positive function
$\theta(c)$ with the property that the largest prime factor
of $\fl{n^c}$ exceeds $n^{\theta(c)-\eps}$ infinitely often.
For $c\in(1,\tfrac{149}{87})$ we show that the counting
function of natural numbers $n\le x$ for which $\fl{n^c}$ is
squarefree satisfies the expected asymptotic formula.
For $c\in(1,\tfrac{147}{145})$ we show that there
are infinitely many Carmichael numbers composed entirely of
primes of the form $p=\fl{n^c}$.
\end{abstract}

\begin{quote}
\end{quote}

\begin{quote}
\end{quote}

\section{Introduction}
Throughout the paper, the integer part of a real number $t$
is denoted by $\fl{t}$.

The \emph{Piatetski-Shapiro sequences} are sequences of the form
$$
\bigl(\fl{n^c}\bigr)_{n\in\N}\qquad (c>1,~c\not\in\N).
$$
They are named in honor of Piatetski-Shapiro, who proved~(cf.~\cite{PS})
that for any number~$c\in(1,\tfrac{12}{11})$ there
are infinitely many primes of the form $\fl{n^c}$.
The admissible range for $c$ in this theorem has been
extended many times over the years, and the result is currently
known for all $c\in(1,\tfrac{243}{205})$ 
(cf.\ Rivat and Wu~\cite{RivatWu}).

In the present paper we examine various arithmetic questions
about the Piatetski-Shapiro sequences.  For instance, denoting
by $P(m)$ the largest prime factor of an integer $m\ge 2$,
we exhibit a positive function $\theta(c)$ which has the property
that, for any non-integer $c>1$ and real $\eps>0$, the inequality
\begin{equation}
\label{(intro)eq:1}
P(\fl{n^c})>n^{\theta(c)-\eps}
\end{equation}
holds for infinitely many $n$.  Our results extend and improve the
earlier work of Abud~\cite{Abud} and of Arkhipov and Chubarikov~\cite{ArkChub}.
The latter authors claim that for any $c\in(1,2)$ one has
$$
P(\fl{n^c})>n^{(27-13c)/28-\eps}
$$
for infinitely many $n$; however, since they do not establish a result
similar to our Proposition~\ref{(dd)prop:1} (see~\S\ref{sec:dd})
to eliminate prime powers $p^k$ with $k\ge 2$, their result cannot be
substantiated for $c\ge\tfrac{149}{87}=1.712\cdots$.  The results
presented here are much sharper than those in \cite{ArkChub} and cover a wider range.

Throughout the paper, we make the convention that if a result is stated in which 
$\varepsilon$ appears, then $\varepsilon$ 
denotes an arbitrary sufficiently small positive number. 

\begin{theorem}
\label{(intro)thm:1}
Let $\theta(c)$ be the piecewise linear function given by
$$
\theta(c)=\begin{cases}
2-c&\quad\text{if $\tfrac{243}{205}\le c<\tfrac{24979}{20803}$;}\\
3-2c&\quad\text{if $\tfrac{24979}{20803}\le c\le\tfrac{112}{87}$;}\\
(92-49c)/68&\quad\text{if $\tfrac{112}{87}\le c\le\tfrac{160}{117}$;}\\
(74-31c)/86&\quad\text{if $\tfrac{160}{117}\le c\le\tfrac{128}{85}$;}\\
(23-10c)/25&\quad\text{if $\tfrac{128}{85}\le c\le\tfrac{31}{20}$;}\\
(4-2c)/3&\quad\text{if $\tfrac{31}{20}\le c\le\tfrac53$;}\\
(3-c)/6&\quad\text{if $\tfrac53\le c<2$.}\\
\end{cases}
$$
Then, for any $c\in[\tfrac{243}{205},2)$
the inequality~\eqref{(intro)eq:1} holds for infinitely many $n$.
\end{theorem}

\begin{theorem}
\label{(intro)thm:2}
There exists a constant $\beta>0$ such that,
for any $c>2$, $c\not\in\N$, the inequality
$$
P(\fl{n^c})>n^{\beta/c^2}
$$
holds for infinitely many $n$.
\end{theorem}

\noindent
Theorem~\ref{(intro)thm:1} is proved in
\S\S\ref{sec:largeprimefactor}--\ref{sec:sv};
Theorem~\ref{(intro)thm:2} is proved in
\S\ref{sec:largeprimefactor}.

The most important tool for our proof of Theorem~\ref{(intro)thm:1}
is the following exponential sum estimate, which is obtained by
adapting the work of Cao and Zhai~\cite{CaoZhai2} (actually,
our result is much simpler in form than that in \cite{CaoZhai2}).
Here and below we use notation like $m\sim M$ as an abbreviation
for $M<m\le 2M$, and $(m_1,\ldots,m_k)\sim(M_1,\ldots,M_k)$ means
that $m_1\sim M_1,\ldots,m_k\sim M_k$.

\begin{theorem}
\label{(intro)thm:3}
Let
$$
S=\sum_{(m,m_1,m_2)\sim(M,M_1,M_2)}
a(m)\,b(m_1,m_2)\,\e(Am^\alpha m_1^\beta m_2^\gamma),
$$
where $M,M_1,M_2\ge 1$, $A\ne 0$, $|a(m)|\le 1$,
$|b(m_1,m_2)|\le 1$, and the constants $\alpha,\beta,\gamma$ satisfy 
$\alpha(\alpha-1)(\alpha-2)\beta\gamma\ne 0$.
Writing $N=M_1M_2$ and $F=|A|M^\alpha M_1^\beta M_2^\gamma$ we have
\begin{align*}
S\,(MN)^{-\eps}&\ll
M^{5/8}N^{7/8}F^{1/8}+MN^{7/8}+M^{37/49}N^{46/49}F^{3/49}\\
&\quad+M^{23/29}N^{27/29}F^{3/58}+M^{43/58}N^{27/29}F^{2/29}
+M^{115/152}N^{7/8}F^{25/304}\\
&\quad+M^{41/54}N^{25/27}F^{7/108}+M^{5/6}N+M^{11/10}NF^{-1/4}.
\end{align*}
\end{theorem}

\noindent This is proved in \S\ref{sec:esm}.

As another application of Theorem~\ref{(intro)thm:3} we give
in \S\ref{sec:dd} a detailed proof of a result sketched
by Cao and Zhai~\cite{CaoZhai3}; their earlier paper~\cite{CaoZhai1}
covers the narrower range $1<c<\tfrac{61}{36}$.

\begin{theorem}
\label{(intro)thm:4}
For fixed $c\in(1,\tfrac{149}{87})$ we have
\begin{equation}
\label{(intro)eq:2}
\#\big\{n\le x:\fl{n^c}\text{~\rm is squarefree}\big\}
=\frac{6}{\pi^2}\,x+O(x^{1-\eps}).
\end{equation}
\end{theorem}

A third application  of Theorem~\ref{(intro)thm:3} is the following
result, which is needed for our proof of Theorem~\ref{(intro)thm:1}
and may be of independent interest; the proof is given in \S\ref{sec:dd}.

\begin{theorem}
\label{(intro)thm:5}
For fixed $c\in(1,\tfrac{149}{87})$ the inequality
$$
\sum_{p\le x^c}\log p\sum_{\substack{n\le x\\ p\,\mid\,\fl{n^c}}}1
>(c-\eps)\,x\log x
$$
holds for all sufficiently large $x$.
\end{theorem}

A question that has not been previously considered is
the following: for which values of $c$ is it true that one has
$$
P(\fl{n^c})\le n^\eps
$$
for infinitely many $n$?  In this paper, we show that
this is the case whenever
\begin{equation}
\label{(intro)eq:3}
1<c<\tfrac{24979}{20803}=1.2007\cdots.
\end{equation}
More precisely, we prove the following result in \S\ref{sec:sv}.

\begin{theorem}
\label{(intro)thm:6}
For any number $c$ in the range~\eqref{(intro)eq:3} we have
$$
\#\big\{n\le x:P(\fl{n^c})\le n^\eps\big\}\gg x^{1-\eps}.
$$
\end{theorem}

%
%

Finally, we consider a problem connected with
\emph{Carmichael numbers}, which are \emph{composite} natural numbers
$N$ with the property that $N\mid a^N-a$ for every $a\in\Z$.
The existence of infinitely many Carmichael numbers was established
in 1994 by Alford, Granville and Pomerance \cite{AGP}.
In \S\ref{sec:carm} we adapt the method of \cite{AGP}
to prove the following result.

\begin{theorem}
\label{(intro)thm:8} For every $c\in\(1,\tfrac{147}{145}\)$
there are infinitely many Carmichael numbers composed entirely
of primes from the set
$$
\sP^{(c)}=\big\{p\text{~\rm prime}:p=\fl{n^c}\text{~\rm for some~}n\in\N\big\}.
$$
\end{theorem}

\noindent We call the members of $\sP^{(c)}$ \emph{Piatetski-Shapiro primes}.
The proof of Theorem~\ref{(intro)thm:8} requires a considerable amount of
information about the distribution of Piatetski-Shapiro primes
in arithmetic progressions.  Here, we single out one such result.
Writing
$$
\pi(x;d,a)=\#\big\{p\le x:p\equiv a\bmod d\big\}
$$
and
$$
\pi_c(x;d,a)=\#\big\{p\le x:p\in\sP^{(c)},~p\equiv a\bmod d\big\},
$$
we establish the following result in \S\ref{sec:carm}.

\begin{theorem}
\label{(intro)thm:9}
Let $a$ and $d$ be coprime integers, $d\ge 1$.
For fixed $c\in\(1,\tfrac{18}{17}\)$ we have
\begin{align*}
\pi_c(x;d,a)&=\gamma x^{\gamma-1}\,\pi(x;d,a)
+\gamma(1-\gamma)\int_2^x u^{\gamma-2}\,\pi(u;d,a)\,du\\
&\quad+O\bigl(x^{17/39+7\gamma/13+\eps}\bigr).
\end{align*}
\end{theorem}

We remark that, for each of the various results
obtained in the present paper, the admissible range of $c$ depends on the
quality of our bounds for certain exponential sums; the particular
type of exponential sum that is needed varies from one application
to the next.

\section{Notation and preliminaries}
\label{sec:notate}

As usual, for all $t\in\R$ we write
$$
\e(t)=e^{2\pi it},\qquad
\|t\|=\min_{n\in\Z}|t-n|,\qquad
\{t\}=t-\fl{t}.
$$
We make considerable use of the sawtooth function
$$
\psi(t)=t-\fl{t}-\tfrac12=\{t\}-\tfrac12
$$
along with the
well known approximation of Vaaler~\cite{Vaal}:
there exist numbers
$c_h~(0<|h|\le H)$ and $d_h~(|h|\le H)$ such that
\begin{equation}
\label{(notate)eq:3}
\biggl|\psi(t)-\sum_{0<|h|\le H}c_h\,\e(th)\biggl|
\le\sum_{|h|\le H}d_h\,\e(th),\quad
c_h\ll\frac{1}{|h|},\quad d_h\ll\frac{1}{H}.
\end{equation}

We  use the following basic exponential sum estimates several
times in the sequel.

\begin{lemma}
\label{(notate)lem:1}
Let $f$ be three times continuously differentiable on a subinterval
$\cI$ of $(N,2N]$.
\begin{itemize}
\item[$(i)$] Suppose that for some $\lambda>0$,
the inequalities
$$
\lambda\ll|f''(t)|\ll\lambda\qquad(t\in\cI)
$$
hold, where the implied constants are independent
of $f$ and $\lambda$. Then
$$
\sum_{n\in\cI}\e(f(n))\ll N\lambda^{1/2}+\lambda^{-1/2}.
$$
\item[$(ii)$] Suppose that for some $\lambda>0$,
the inequalities
$$
\lambda\ll|f'''(t)|\ll\lambda\qquad(t\in\cI)
$$
hold, where the implied constants are independent
of $f$ and $\lambda$. Then
$$
\sum_{n\in\cI}\e(f(n))\ll N\lambda^{1/6}+
N^{3/4}+N^{1/4}\lambda^{-1/4}.
$$
\end{itemize}
\end{lemma}

\begin{proof}
See Graham and Kolesnik \cite[Theorems~2.2 and~2.6]{GraKol}.
\end{proof}

\begin{lemma}
\label{(notate)lem:2}
Fix $c\in(1,2)$, and put $\gamma=1/c$.
Let $z_1,z_2,\ldots$ be complex numbers such that
$z_k\ll k^\eps$. Then
$$
\sum_{\substack{k\le K\\k=\fl{n^c}}}z_k
=\gamma\sum_{k\le K}z_k k^{\gamma-1}
+\sum_{k\le K}z_k\bigl(\psi(-(k+1)^\gamma)-\psi(-k^\gamma)\bigr)+O(1).
$$
\end{lemma}

\begin{proof}
The equality $k=\fl{n^c}$ holds precisely when $k\le n^c<k+1$,
or equivalently, when $-(k+1)^\gamma\le -n<-k^\gamma$.  
Consequently,
\begin{align*}
\sum_{\substack{k\le K\\k=\fl{n^c}}}z_k
&=\sum_{k\le K}z_k\bigl(\fl{-k^\gamma}-\fl{-(k+1)^\gamma}\bigr)\\
&=\sum_{k\le K}z_k\bigl((k+1)^\gamma-k^\gamma\bigr)
+\sum_{k\le K}z_k\bigl(\psi(-(k+1)^\gamma)-\psi(-k^\gamma)\bigr).
\end{align*}
The result now follows on applying the mean value theorem and
taking into account that $\sum_{k\le K}|z_k|k^{\gamma-2}\ll 1$.
\end{proof}

\begin{lemma}
\label{(notate)lem:3}
{\rm (Erd\H os-Tur\'an)}
Let $t_1,\ldots,t_K\in\R$, $\beta\in(0,1)$, and $H\ge 1$.  Then
$$
\#\big\{k\le K:\{t_k\}\le\beta\big\}-K\beta
\ll\frac{K}{H}+\sum_{h\le H}\frac1h\biggl|\sum_{k=1}^K\e(t_k h)\biggl|.
$$
\end{lemma}

\begin{proof}
See Baker~\cite[Theorem~2.1]{Ba1986}.
\end{proof}

We need a simple ``decomposition result'' for sums of the form
$$
\sum_{X<n\le X_1}\Lambda(n)f(n),
$$
where $f$ is any complex-valued function, and $X_1\sim X$.
A \emph{Type I sum} is a sum of the form
$$
S_I=\mathop{\sum_{k\sim K}\sum_{\ell\sim L}}\limits_{X<k\ell\le X_1}
a_k\,f(k\ell)
$$
in which $|a_k|\le 1$ for all $k\sim K$. A \emph{Type II sum}
is a sum of the form
\begin{equation}
\label{(intro)eq:typeII}
S_{I\!I}=\mathop{\sum_{k\sim K}~\sum_{\ell\sim L}}\limits_{X<k\ell\le X_1}
a_k\,b_\ell\,f(k\ell)
\end{equation}
in which $|a_k|\le 1$ and $|b_\ell|\le 1$ for all $(k,\ell)\sim(K,L)$.
The following result can be derived from \emph{Vaughan's identity}
(see Vaughan~\cite{Vau} or Davenport~\cite[Chapter~15]{Dav}).

\begin{lemma}
\label{(notate)lem:4}
Suppose that every Type~I sum with $L\gg X^{2/3}$ satisfies the bound
$$
S_I\ll B(X)
$$
and that every Type~II sum with $X^{1/3}\ll K\ll X^{1/2}$ satisfies
the bound
$$
S_{I\!I}\ll B(X).
$$
Then
$$
\sum_{X<n\le X_1}\Lambda(n)f(n)\ll B(X)X^\eps.
$$
\end{lemma}

A standard procedure for estimating Type~II sums with functions
of the form $f(n)=\e(g(n))$ can be derived from the proof of
\cite[Lemma~4.13]{GraKol}.

\begin{lemma}
\label{(notate)lem:5}
Let $1<Q\le L$.  If $f$ is a function of the form $f(n)=\e(g(n))$,
then any Type~II sum~\eqref{(intro)eq:typeII} satisfies
$$
|S_{I\!I}|^2\ll X^2Q^{-1}+XQ^{-1}
\sum_{0<|q|<Q}\sum_{\ell\sim L}\big|S(q,\ell)\big|,
$$
where
$$
S(q,\ell)=\sum_{k\in\cI(q,\ell)}\e\bigl(g(k\ell)-g(k(\ell+q))\bigr)
$$
for a certain subinterval $\cI(q,\ell)$ of $(X,X_1]$.
\end{lemma}

\section{Exponential sums with monomials}
\label{sec:esm}

Theorem~\ref{(intro)thm:3} is proved via the method of
Cao and Zhai~\cite{CaoZhai2}. The upper bound in our theorem
has nine terms, whereas in \cite[Theorem~6]{CaoZhai2}
the corresponding upper bound has fourteen terms.  Since Cao and Zhai
omit the details of their optimization, we do not know how our
optimization differs from theirs.

For the proof, we require four general results
from the literature, which are reproduced here for the convenience
of the reader; some other results are quoted during the course
of the proof.

\begin{lemma}
\label{(esm)lem:1}
Let $Y=(y_k)_{k\sim K}$ and $Z=(z_\ell)_{\ell\sim L}$ be two sequences
of complex numbers with $|y_k|\le 1$, $|z_\ell|\le 1$.
Let $\alpha_k,\beta_\ell\in\C$, and put
$$
S_{\alpha,\beta}(Y,Z)=\sum_{k\sim K}\sum_{\ell\sim L}
\alpha_k\,\beta_\ell\,\e(B y_kz_\ell).
$$
Then
$$
\big|S_{\alpha,\beta}(Y,Z)\big|^2\le 20(1+B)\,S_\alpha(Y,B^{-1})\,S_\beta(Z,B^{-1}),
$$
where 
$$
S_\alpha(Y,B^{-1})
=\sum_{\substack{k,k'\sim K\\|y_k-y_{k'}|\le B^{-1}}}|\alpha_k\alpha_{k'}|
\mand
S_\beta(Z,B^{-1})
=\sum_{\substack{\ell,\ell'\sim L\\|z_\ell-z_{\ell'}|\le B^{-1}}}
|\beta_\ell\beta_{\ell'}|.
$$
\end{lemma}

\begin{proof}
See Bombieri and Iwaniec~\cite[Lemma~2.4]{BomIwa}.
\end{proof}

\begin{lemma}
\label{(esm)lem:2}
Let $\alpha,\beta\in\R$ with $\alpha\beta\ne 0$, and let
$K,L\ge 1$. Put
$$
u(k,\ell)=\frac{k^\alpha \ell^\beta}{K^\alpha L^\beta}
\qquad(k\sim K,~\ell\sim L).
$$
Then, for any $C>0$ we have
\begin{align*}
&\#\big\{(k,\widetilde k,\ell,\widetilde \ell\,):
k,\widetilde k\sim K,~\ell,\widetilde \ell\sim L,~
|u(k,\ell)-u(\widetilde k,\widetilde \ell\,)|\le C\big\}\hfil\\
&\qquad\ll KL\log(2KL)+K^2L^2C.
\end{align*}
\end{lemma}

\begin{proof}
See Fouvry and Iwaniec~\cite[Lemma~1]{FouIwa}.
\end{proof}

\begin{lemma}
\label{(esm)lem:3}
Let $N,Q\ge 1$, and let $Z=(z_n)_{n\sim N}$ be a sequence
of complex numbers.  Then
$$
\biggl|\,\sum_{n\sim N}z_n\biggl|^2\le
\(2+\frac{N}{Q}\)\sum_{|q|\le Q}\(1-\frac{|q|}{Q}\)
\sum_{n:\,N<n\pm q\le 2N}z_{n+q}\overline z_{n-q}.
$$
\end{lemma}

\begin{proof}
See~\cite[Lemma~2]{FouIwa}.
\end{proof}

\begin{lemma}
\label{(esm)lem:4}
Let
$$
L(Q)=\sum_{j=1}^J C_j Q^{c_j}+\sum_{k=1}^K D_k Q^{-d_k},
$$
where $C_j,c_j,D_k,d_k>0$.  Then
\begin{itemize}

\item[$(i)$] For any $Q\ge Q'>0$ there exists $Q_1\in[Q',Q]$ such that
$$
L(Q_1)\ll \sum_{j=1}^J \sum_{k=1}^K\bigl(C_j^{d_k} D_k^{c_j}\bigr)^{1/(c_j+d_k)}
+\sum_{j=1}^J C_j(Q')^{c_j}+\sum_{k=1}^K D_k Q^{-d_k}.
$$

\item[$(ii)$] For any $Q>0$ there exists $Q_1\in(0,Q]$ such that
$$
L(Q_1)\ll \sum_{j=1}^J \sum_{k=1}^K\bigl(C_j^{d_k} D_k^{c_j}\bigr)^{1/(c_j+d_k)}
+\sum_{k=1}^K D_k Q^{-d_k}.
$$
\end{itemize}
\end{lemma}

\begin{proof}
See~\cite[Lemma~2.4]{GraKol} for a proof of the first
assertion; the second assertion can be proved similarly.
\end{proof}

\begin{proof}[Proof of Theorem~\ref{(intro)thm:3}]
Let $T_1,T_2,\ldots,T_9$ respectively denote the nine terms in the bound
of the theorem.

Applying~\cite[Theorem~3]{FouIwa} we have the bound
$$
S\sL^{-2}\ll M^{1/2}N^{3/4}F^{1/4}+M^{7/10}N+MN^{3/4}+M^{11/10}NF^{-1/4},
$$
where $\sL=\log(2MN)$.  In the case that $F\le M^2N^{1/2}$ it follows that
$$
S\sL^{-2}\ll T_2+T_8+T_9,
$$
and the theorem is proved; thus, we suppose from now on 
that $F\ge M^2N^{1/2}$.

By Cauchy's inequality we have
$$
|S|^2\le N\sum_{\substack{m_1\sim M_1\\ m_2\sim M_2}}
\left|\sum_{m\sim M}a(m)\,\e(Am^\alpha m_1^\beta m_2^\gamma)\right|^2.
$$

Let $Q$ be a parameter (to be optimized later) such that $10\le Q\le M^{1/3-\eps}$.
Applying Lemma~\ref{(esm)lem:3} to the inner sum, we obtain (after splitting
the range of~$q$ into dyadic subintervals)
\begin{equation}
\label{(esm)eq:1}
|S|^2\sL^{-1}\ll M^2N^2Q^{-1}+MNQ^{-1}\Sigma,
\end{equation}
where
$$
\Sigma=\sum_{\substack{m_1\sim M_1\\ m_2\sim M_2}}
\sum_{\substack{m\sim M\\ q_1\sim Q_1}}
c(m,q_1)\,\e(t(m,q_1)\,Am_1^\beta m_2^\gamma)
$$
for some $Q_1\in[\tfrac12,Q]$, with
\begin{align*}
c(m,q_1)&=a(m+q_1)\,\overline{a(m-q_1)},\\
t(m,q_1)&=(m+q_1)^\alpha-(m-q_1)^\alpha.
\end{align*}
Note that $|c(m,q_1)|\le 1$ for all $(m,q_1)\sim(M,Q_1)$.

Next, we put $Q_2=Q_1^2$ and again apply Cauchy's inequality,
Lemma~\ref{(esm)lem:3} and a dyadic splitting argument to derive
the bound
\begin{equation}
\label{(esm)eq:2}
\sL^{-1}\Sigma^2\ll M^2N^2Q_1^2Q_2^{-1}+MNQ_1Q_2^{-1}\Sigma_1
=M^2N^2+MNQ_1^{-1}\Sigma_1,
\end{equation}
where
$$
\Sigma_1=\sum_{\substack{m_1\sim M_1\\ m_2\sim M_2}}
\sum_{\substack{m\sim M\\ q_1\sim Q_1\\ q_2\sim Q_2^*}}
c(m,q_1,q_2)\,\e(t(m,q_1,q_2)\,Am_1^\beta m_2^\gamma)
$$
for some $Q_2^*\in[\tfrac12,Q_2]$, with
\begin{align*}
c(m,q_1,q_2)&=c(m+q_2,q_1)\,\overline{c(m-q_2,q_1)},\\
t(m,q_1,q_2)&=t(m+q_2,q_1)-t(m-q_2,q_1).
\end{align*}
Note that $|c(m,q_1,q_2)|\le 1$ for all $(m,q_1,q_2)\sim(M,Q_1,Q_2)$.

We now partition the sum $\Sigma_1$. To do this, we put
$$
Q_a=\min\{Q_1,Q_2^*\}\mand Q_b=\max\{Q_1,Q_2^*\}.
$$
Let $f$ be the function defined by
$$
f(q_1,q_2)=(q_1q_2^{\alpha-1})^{1/(\alpha-2)},
$$
and let $c'>c>0$ be suitable constants (depending only on $\alpha$)
such that the interval
$$
\cI=\bigr[c\,f(Q_a,Q_b),c'\,f(Q_a,Q_b)\bigr]
$$
contains all numbers of the form $f(q_1,q_2)$
with $(q_1,q_2)\sim(Q_a,Q_b)$.  Let $\eta$ be selected from the range
\begin{equation}
\label{(esm)eq:3}
\max\big\{Q_a^2Q_b^{-2},3\sL Q_a^{-1}Q_b^{-1}\big\}
\le \eta\le c'/c-1.
\end{equation}
Let $a_k=(1+\eta)^k\,c\,f(Q_a,Q_b)$
and $\cI_k=[a_k,(1+\eta)a_k]$ for $0\le k\le K$, where
$$
K=\fl{\frac{\log(c'/c)}{\log(1+\eta)}}.
$$
Note that $K\asymp\eta^{-1}$ for all $\eta$ satisfying~\eqref{(esm)eq:3}.
Since
$t(m,q_1,q_2)=t(m,q_2,q_1)$ we have
$$
\Sigma_1=\sum_{0\le k\le K}\sum_{\substack{(q_1,q_2)\sim(Q_a,Q_b)\\
f(q_1,q_2)\in\cI_k}}
\sum_{\substack{m\sim M\\m_1\sim M_1\\m_2\sim M_2}}
\e(t(m,q_1,q_2)\,Am_1^\beta m_2^\gamma).
$$
Let $D_k$ be the number of 6-tuples
$(m,\widetilde m,q_1,\widetilde q_1,q_2,\widetilde q_2)\sim(M,M,Q_a,Q_a,Q_b,Q_b)$
such that $f(q_1,q_2)$ and $f(\widetilde q_1,\widetilde q_2)$ lie in $\cI_k$ and
$$
\big|t(m,q_1,q_2)-t(\widetilde m,\widetilde q_1,\widetilde q_2)\big|\ll 
\frac{1}{|A|M_1^\beta M_2^\gamma}\,,
$$
and let $E$ be the number of 4-tuples
$(m_1,\widetilde m_1,m_2,\widetilde m_2)\sim(M_1,M_1,M_2,M_2)$ such that
$$
\big|m_1^\beta m_2^\gamma-\widetilde m_1^\beta\widetilde m_2^\gamma\big|
\ll\frac{1}{|A|M^{\alpha-2}Q_1Q_2^*}\,.
$$
An application of Lemma~\ref{(esm)lem:1} for each value of $k$
(taking $B\asymp M^{-2}FQ_1Q_2^*$
and using the fact that $F\ge M^2N^{1/2}$) yields the bound
$$
\Sigma_1\ll (M^{-2}FQ_1Q_2^*E)^{1/2}\sum_{0\le k\le K}D_k^{1/2}.
$$
Using Cauchy's inequality again we have
\begin{equation}
\label{(esm)eq:4}
\Sigma_1^2\ll M^{-2}FQ_1Q_2^*E\eta^{-1}\sum_{0\le k\le K}D_k.
\end{equation}

First assume that $\alpha \ne 3$.

If $Q_b>Q_a M^{\eps/4}$ we are in a position to 
apply~\cite[Theorem~2]{CaoZhai2};
the conditions $Q_b\le M^{1-\eps}$ and $Q_aQ_b\le M^{3/2-\eps}$
are certainly satisfied.  For a suitably chosen $\eta$
satisfying~\eqref{(esm)eq:3} we obtain the bound
\begin{equation}
\label{(esm)eq:5}
M^{-\eps}\eta^{-1}\sum_{0\le k\le K}D_k\ll B_1,
\end{equation}
where 
\begin{align*}
B_1&=MQ_aQ_b+M^4  F^{-1}Q_aQ_b
+M^{1/4}Q_a^{7/4}Q_b^{9/4}+M^{-2}Q_a^4Q_b^4\\
&\quad+M^{3/4}F^{-1/8}Q_a^{7/4}Q_b^{2}
+M^3 F^{-1/2}Q_a +Q_a^{13/6}Q_b^{5/2}\\
&\quad+MF^{-1/4}Q_a^{7/4} Q_b^{9/4}
+M^{-1/2}Q_a^{5/2} Q_b^3.
\end{align*}
In the case that $Q_b\le Q_a M^{\eps/4}$ we 
apply~\cite[Theorem~1]{CaoZhai2} with the choices $K=0$
and $\eta=c'/c$.  Since the condition $Q_b\le M^{2/3-\eps}$
is clearly satisfied, we see that
$$
M^{-\eps/2}\eta^{-1}D_0\ll M^{-\eps/2}D_0
\ll MQ_aQ_b+M^4Q_aQ_bF^{-1}+M^{-2}Q_a^2Q_b^6+Q_a^2Q_b^{8/3}.
$$
Since
$$
M^{-2}Q_a^2Q_b^6\le M^{-2}Q_a^4Q_b^4\cdot M^{\eps/2}
\mand Q_a^2Q_b^{8/3}\le Q_a^{13/6}Q_b^{5/2}\cdot M^{\eps/2},
$$
we obtain~\eqref{(esm)eq:5} in this case as well.

Since $Q_a\le Q_1$ and $Q_b\le Q_2=Q_1^2$ we find that
\begin{equation}
\label{(esm)eq:6}
M^{-\eps}\eta^{-1}\sum_{0\le k\le K}D_k\ll B_2,
\end{equation}
where
\begin{align*}
B_2&=MQ_1^3+M^4F^{-1}Q_1^3+M^{1/4}Q_1^{25/4}+ M^{-2}Q_1^{12}+M^{3/4}F^{-1/8}Q_1^{23/4}\\
&\quad+M^3F^{-1/2}Q_1+ Q_1^{43/6}+MF^{-1/4}Q_1^{25/4}
+M^{-1/2}Q_1^{17/2}.
\end{align*}

We now notice that for $\alpha = 3$ we have $ t(m,q_1,q_2) = 24mq_1q_2$,
so the bound~\eqref{(esm)eq:6} is immediate in this case. 

To bound $E$ we use Lemma~\ref{(esm)lem:2} to derive that
\begin{equation}
\label{(esm)eq:7}
E\ll N\sL+\frac{M^2N^2}{FQ_1Q_2^*}\,.
\end{equation}
Combining~\eqref{(esm)eq:4}, \eqref{(esm)eq:6} and~\eqref{(esm)eq:7}, it follows that
$$
 M^{-2\eps}\Sigma_1^2
\ll M^{-2}FQ_1Q_2^*\bigl(N+M^2N^2/(FQ_1Q_2^*)\bigr)B_2
\le(M^{-2}NFQ_1^3+N^2)B_2.
$$
Taking into account~\eqref{(esm)eq:2} we see that
\begin{align*}
 M^{-3\eps}\Sigma^4
&\ll M^4N^4+M^2N^2Q_1^{-2}\cdot M^{-2\eps}\Sigma_1^2 \\
&\ll M^4N^4+(FN^3Q_1+M^2N^4Q_1^{-2})B_2.
\end{align*}
In the last expression only one term has a negative exponent
of $Q_1$, namely,
$$
(M^2N^4Q_1^{-2})(M^3F^{-1/2}Q_1)\ll M^5N^4F^{-1/2};
$$
in the other terms, we replace $Q_1$ by $Q$. In view 
of~\eqref{(esm)eq:1} we derive the bound
\begin{align*}
|S|^8 M^{-4\eps}
&\ll M^8N^8Q^{-4}+M^4N^4Q^{-4}\cdot M^{-3\eps}\Sigma^4\\
&\ll M^8N^8Q^{-4}+M^5N^7F+M^8N^7+M^{17/4}N^7FQ^{13/4}\\
&\quad+M^2N^7FQ^9+M^{19/4}N^7F^{7/8}Q^{11/4}
+M^7N^7F^{1/2}Q^{-2}\\
&\quad+M^4N^7FQ^{25/6}+M^5N^7F^{3/4}Q^{13/4}
+M^{7/2}N^7FQ^{11/2}\\
&\quad+M^7N^8Q^{-3}+M^{10}N^8F^{-1}Q^{-3}+M^{25/4}N^8Q^{1/4}\\
&\quad+M^4N^8Q^6+M^{27/4}N^8F^{-1/8}Q^{-1/4}+M^9N^8F^{-1/2}Q^{-4}\\
&\quad+M^6N^8Q^{7/6}+M^7N^8F^{-1/4}Q^{1/4}+M^{11/2}N^8Q^{5/2}\\
&=U_1+U_2+\cdots+U_{19}\qquad\text{(say)}.
\end{align*}

Because $F\ge M^2$ and $Q\le M^{1/3}$,
we can discard $U_{15}$ and $U_{18}$ in view of the term $M^{5/6}N$
in the bound of Theorem~\ref{(intro)thm:3}.  Collecting
terms for which the exponent of $F$ is $1$, we use $Q\le M^{1/3}$ to
eliminate $U_5$ and $U_{10}$:
\begin{align*}
U_5\le U_2\mand U_{10}\le U_8.
\end{align*}
Collecting terms in which $F$ is absent, we use $Q\le M^{1/3}$ to
eliminate $U_{11}$, $U_{13}$, $U_{14}$, $U_{17}$ and $U_{19}$:
$$
\max\{U_{11}, U_{13}, U_{14}, U_{17}, U_{19}\}\le U_1.
$$
We can also discard the term $U_{16}$ since the bound
$U_{16}\ll U_1$ follows from the inequalities
$F\ge M^2$ and $Q\ge\tfrac12$.  Finally, the term $U_{12}$
can be eliminated as the inequality $F\ge M^2N^{1/2}$ implies
that
$$
U_{12}=M^{10}N^8F^{-1}Q^{-3}
\le(M^8N^8Q^{-4})^{1/2}(M^7N^7F^{1/2}Q^{-2})^{1/2}
=(U_1U_7)^{1/2}.
$$
After eliminating these terms, we are left with the bound
\begin{align*}
|S|^8M^{-4\eps}
&\ll M^4N^7FQ^{25/6}
+(M^{17/4}N^7F+M^5N^7F^{3/4})Q^{13/4}\\
&\quad+M^{19/4}N^7F^{7/8}Q^{11/4}
+M^5N^7F+M^8N^7\\
&\quad+M^7N^7F^{1/2}Q^{-2}
+M^8N^8Q^{-4}.
\end{align*}
Now we apply Lemma~\ref{(esm)lem:4} to derive that
\begin{align*}
|S|^8M^{-4\eps}&\ll
M^5N^7F+M^8N^7+M^{223/37}N^7F^{49/74}+M^{296/49}N^{368/49}F^{24/49}\\
&\quad+M^{131/21}N^7F^{25/42}+M^{184/29}N^{216/29}F^{12/29}+M^{125/21}N^7F^{29/42}\\
&\quad+M^{172/29}N^{216/29}F^{16/29}+M^{115/19}N^7F^{25/38}\\
&\quad+M^{164/27}N^{200/27}F^{14/27} +M^4N^7F+M^5N^7F^{3/4}+M^{17/4}N^7F\\
&\quad+M^{19/4}N^7F^{7/8} +M^{19/3}N^7F^{1/2}+M^{20/3}N^8\\
&=V_1+V_2+\cdots+V_{16}\qquad\text{(say)}.
\end{align*}
We can discard half of these terms using the following facts:
\begin{itemize}
\item[$(i)$] $V_3\le(V_1^{52}V_2^{22}V_9^{703})^{1/777}$;
\item[$(ii)$] $V_5=V_2^{2/21}V_9^{19/21}$;
\item[$(iii)$] $V_7=V_1^{2/21}V_9^{19/21}$;
\item[$(iv)$] $\max\{V_{11}, V_{12}, V_{13}, V_{14}\}\le V_1$;
\item[$(v)$] $V_{15}\le V_1^{1/2}V_2^{1/2}$.
\end{itemize}
Therefore, we arrive at the bound
\begin{equation*}
\begin{split}
|S_I|^8M^{-4\eps}&\ll V_1+V_2+V_4+V_6+V_8+V_9+V_{10}+V_{16}\\
&=T_1^8+T_2^8+T_3^8+T_4^8+T_5^8+T_6^8+T_7^8+T_8^8,
\end{split}
\end{equation*}
as required.
\end{proof}

\section{On the divisibility of $\fl{n^c}$ by squares}
\label{sec:dd}

The following proposition is needed for the proofs
of Theorems~\ref{(intro)thm:4} and~\ref{(intro)thm:5}.

\begin{proposition}
\label{(dd)prop:1}
Fix $c\in(1,\tfrac{149}{87})$. 
Let $1\le D\le x^{c/2}$, and let $(z_d)_{d\sim D}$ be
a sequence of complex numbers such that $z_d\ll\log d$.  Then
\begin{equation}
\label{(dd)eq:1}
\sum_{d\sim D}z_d\sum_{\substack{n\le x\\d^2\,\mid\,\fl{n^c}}}1
=x\sum_{d\sim D}\frac{z_d}{d^2}+O(x^{1-\eps}).
\end{equation}
\end{proposition}

\begin{proof}
First, suppose that $D\le x^{2-c-6\eps}$.
Let $S_d$ be the inner sum on the left-hand
side of~\eqref{(dd)eq:1}.  By the argument used to
prove Lemma~\ref{(notate)lem:2}, we see that
\begin{align*}
S_d&=\sum_{\ell\le x^c/d^2}\bigl(\fl{-(d^2\ell)^\gamma}
-\fl{-(d^2\ell+1)^\gamma}\bigr)+O(1)\\
&=\sum_{\ell\le x^c/d^2}\bigl((d^2\ell+1)^\gamma-(d^2\ell)^\gamma\bigr)
-\sum_{\ell\le x^c/d^2}\psi(-(d^2\ell)^\gamma)\\
&\quad+\sum_{\ell\le x^c/d^2}\psi(-(d^2\ell+1)^\gamma)+O(1).
\end{align*}
The mean value theorem yields the estimate
$$
\sum_{\ell\le x^c/d^2}\bigl((d^2\ell+1)^\gamma-(d^2\ell)^\gamma\bigr)
=\gamma d^{\gamma-2}\sum_{\ell\le x^c/d^2}\ell^{\gamma-1}+O(1)
=\frac{x}{d^2}+O(1)
$$
(see, e.g., LeVeque~\cite[pp.$\,$138--139]{LeVeque} for the last step).
Hence, to finish the proof in this case it suffices to show that the
bound
\begin{equation}
\label{(dd)eq:2}
\sum_{\ell\le x^c/d^2}\psi(-d^{2\gamma}(\ell+\xi)^\gamma)\ll
D^{-1}x^{1-2\eps}
\end{equation}
holds uniformly for $0\le\xi<1$. Applying~\cite[Lemma~3]{CaoZhai1}
with $\kappa=\lambda=\tfrac12$, the left-hand side of~\eqref{(dd)eq:2} is
\begin{equation*}
\begin{split}
\sum_{\ell\le x^c/d^2}\psi(-d^{2\gamma}(\ell+\xi)^\gamma)&\ll d^{2\gamma/3}(x^c/d^2)^{(1+\gamma)/3}+
d^{-2\gamma}(x^c/d^2)^{1-\gamma}\\
&\ll D^{-2/3}x^{(c+1)/3}+D^{-2}x^{1-\gamma}\\
&\ll D^{-1}x^{1-2\eps},
\end{split}
\end{equation*}
where we have used the inequality $D\le x^{2-c-6\eps}$ in the last step.

Next, we consider the case $D\ge x^{2-c-6\eps}$.
It suffices to show that the sum
\begin{equation}
\label{(dd)eq:3}
S(D,L)=\sum_{d\sim D}\sum_{\ell\sim L}
\bigl(\fl{-(d^2\ell)^\gamma}-\fl{-(d^2\ell+1)^\gamma}\bigr)
\end{equation}
satisfies the bound
$$
S(D,L)\ll x^{1-3\eps}
$$
uniformly for all $L\ge 1$, $D^2L\le x^c$.  
Noting that the summand in~\eqref{(dd)eq:3}
is always either $0$ or $1$, and it is $0$ whenever
$$
\{-(d^2\ell)^\gamma\}>(d^2\ell+1)^\gamma-(d^2\ell)^\gamma,
$$
an application of Lemma~\ref{(notate)lem:3} yields the bound
\begin{equation*}
\begin{split}
S(D,L)&\le \mathop{\sum_{d\sim D}\sum_{\ell\sim L}}
\limits_{\{-(d^2\ell)^\gamma\}\le(D^2L)^{\gamma-1}}1\\
&\ll DL(D^2L)^{\gamma-1}+\frac{DL}{H_1}+
\sum_{h\le H_1}\frac{1}{h}\left|\sum_{d\sim D}\sum_{\ell\sim L}
\e\bigl(h(d^2\ell)^\gamma\bigr)\right|
\end{split}
\end{equation*}
for any number $H_1\ge 1$; we choose $H_1=DL x^{-1+3\eps}$.
Since
$$
DL(D^2L)^{\gamma-1}=D^{-1}(D^2L)^\gamma\ll D^{-1}x\ll x^{1-3\eps},
$$
we need only show that for $\tfrac12\le H<H_1$ and any
sequence $(b_h)_{h\sim H}$ of complex numbers with $|b_h|\le 1$,
the following bound holds uniformly:
\begin{equation}
\label{(dd)eq:4}
S^*=\sum_{h\sim H}b_h\sum_{d\sim D}\sum_{\ell\sim L}\e(h(d^2\ell)^\gamma)
\ll Hx^{1-3\eps}.
\end{equation}

If it is the case that $D>x^{2c-3+16\eps}$ we can deduce~\eqref{(dd)eq:4}
from Robert and Sargos~\cite[Theorem~3]{RoSa}, which yields
\begin{equation}
\label{(dd)eq:5}
S^*\ll x^\eps DLH\(\(\frac{F}{DL^2H}\)^{1/4}+L^{-1/2}+F^{-1}\),
\end{equation}
where
\begin{equation}
\label{(dd)eq:6}
F=H(D^2L)^\gamma\le Hx.
\end{equation}
The second and third summands in~\eqref{(dd)eq:5} are easily dispatched. Indeed,
$$
 DL^{1/2}Hx^\eps\ll Hx^{c/2+\eps}\ll Hx^{1-3\eps},
$$
and
\begin{equation}
\label{(dd)eq:7}
 DLHF^{-1}x^\eps\ll (D^2L)^{1-\gamma}x^\eps\ll x^{c-1+\eps}\ll Hx^{1-3\eps}.
\end{equation}
Taking into account~\eqref{(dd)eq:6} and the inequality $D>x^{2c-3+16\eps}$,
we have for the first summand in~\eqref{(dd)eq:5}:
\begin{align*}
DLH\(\frac{F}{DL^2H}\)^{1/4}x^\eps 
&=(D^2L)^{1/2}D^{-1/4}H^{3/4}F^{1/4}x^\eps \\
&\le (x^c)^{1/2}(x^{2c-3+16\eps})^{-1/4}H^{3/4}(Hx)^{1/4}x^\eps = Hx^{1-3\eps},
\end{align*}
which gives~\eqref{(dd)eq:4} and finishes the proof in this case.

We treat the remaining case $x^{2-c-6\eps}<D\le x^{2c-3+16\eps}$
using Theorem~\ref{(intro)thm:3}.  Let $\eta\asymp 1$ be a real number
such that for $F\le\eta L$ the derivative of the function
$\ell\mapsto h(d^2\ell)^\gamma$ has absolute value at most $1/2$
for $h\sim H$, $d\sim D$.  If $F\le\eta L$,
the Kusmin-Landau inequality (cf.~\cite[Theorem~2.1]{GraKol}) gives
$$
S^*\ll DLHF^{-1},
$$
and the proof is completed using the estimate~\eqref{(dd)eq:7}.
Now suppose that $F\ge\eta L$.
We apply the $B$-process to the sum over $\ell$ in $S^*$.
Following the argument that yields \cite[(6.10)]{RoSa} we have
\begin{align*}
S^*&\ll\frac{L}{F^{1/2}}\int_{-1/2}^{1/2}
\Biggl|\,\sum_{h\sim H}\sum_{d\sim D}\sum_{V<\nu\le V_1}\e(\nu t)\,
\e\(\frac{Y h^{\overline\beta} d^{\overline\gamma}
\nu^{\overline\alpha}}{H^{\overline\beta} D^{\overline\gamma} V^{\overline\alpha}}\)
\Biggl|\,\min\big\{L,|t|^{-1}\big\}\,dt\\
&\quad+DLHF^{-1/2}+DH\log D,
\end{align*}
where
$$
V\asymp V_1\asymp F/L,\qquad
Y\asymp F,\qquad
\overline\beta=\frac{1}{1-\gamma},\qquad
\overline\gamma=\frac{2\gamma}{1-\gamma},\qquad
\overline\alpha=\frac{\gamma}{1-\gamma}.
$$
It is easy to see that
$$
DLHF^{-1/2}=D^{-1}(D^2L)^{1-\gamma/2}H^{1/2}
\ll H^{1/2}x^{2c-5/2+6\eps}
$$
since $D^2L\le x^c$ and $D>x^{2-c-6\eps}$, and that
$$
DH\log D\ll Hx^{2c-3+17\eps}
$$
since $D\le x^{2c-3+16\eps}$. Taking into account that
$c<\frac74$ we obtain the bound
$$
DLHF^{-1/2}+DH\log D\ll Hx^{1-3\eps},
$$
which is acceptable with regards to~\eqref{(dd)eq:4}.
To bound the integrand above, we apply Theorem~\ref{(intro)thm:3} pointwise
with $(F/L,DH)$ instead of $(M,N)$; as a result, it suffices to show that
$$
(F/L)^{5/8}(DH)^{7/8}F^{1/8}+\cdots+(F/L)^{11/10}(DH)F^{-1/4}
\ll(F^{1/2}/L)Hx^{1-4\eps}.
$$
Replacing $F$ by $H(D^2L)^\gamma$, we  now obtain nine separate bounds of the form
\begin{equation}
\label{(dd)eq:8}
D^rL^sH^t(D^2L)^{\gamma u}\ll x^{v-C\eps},
\end{equation}
where $C$ is a positive constant (not necessarily the same at each
occurrence) and the numbers $r,t,s,u,v$ satisfy
$$
t\ge 0,\qquad s+t\ge 0,\qquad u\ge 0,\qquad r\ge 2s+t.
$$
Indeed, using the inequalities $H\le DLx^{-1+3\eps}$, $D^2L\le x^c$,
and $D\le x^{2c-3+16\eps}$, the left-hand side of~\eqref{(dd)eq:8} is
\begin{align*}
D^rL^sH^t(D^2L)^{\gamma u} &\le D^{r+t}L^{s+t}(D^2L)^{\gamma u}x^{-t+3t\eps}
=D^{r-2s-t}(D^2L)^{s+t+\gamma u}x^{-t+3t\eps}\\
&\le (x^{2c-3+16\eps})^{r-2s-t}(x^c)^{s+t+\gamma u}x^{-t+3t\eps}
\ll x^{v-C\eps}
\end{align*}
provided that
$$
(2c-3)(r-2s-t)+c(s+t)<t-u+v.
$$
This leads to the bound
$$
c<\min\big\{\tfrac74,\tfrac{19}{11},\tfrac{149}{87},\tfrac{12}7,\tfrac{85}{49},
\tfrac{163}{95},\tfrac{71}{39}\big\}=\tfrac{149}{87},
$$
and the proof is complete.
\end{proof}

\begin{proof}[Proof of Theorem~\ref{(intro)thm:4}]
Using Proposition~\ref{(dd)prop:1} and a dyadic splitting argument,
the left-hand side of~\eqref{(intro)eq:2} is equal to
\begin{equation*}
\begin{split}
\sum_{n\le x}\sum_{d^2\,\mid\,\fl{n^c}}\mu(d)
=\sum_{d\le x^{c/2}}\mu(d)\sum_{\substack{n\le x\\\fl{n^c}\equiv 0\pmod{d^2}}}1
=x\sum_{d\le x^{c/2}}\frac{\mu(d)}{d^2}+O(x^{1-\eps}).
\end{split}
\end{equation*}
The theorem then follows by extending the series to infinity.
\end{proof}

Next, we turn to the proof of Theorem~\ref{(intro)thm:5},
which eliminates $p^k$ with $k\ge 2$ from a Chebyshev-style
approach to establishing a lower bound for $P(\fl{n^c})$.

\begin{proof}[Proof of Theorem~\ref{(intro)thm:5}]
Clearly,
\begin{equation}
\label{(dd)eq:9}
\sum_{n\le x}\log\fl{n^c}\sim cx\log x.
\end{equation}
The left-hand side of~\eqref{(dd)eq:9} may also be written as
\begin{equation*}
\sum_{n\le x}\sum_{d\,\mid\,\fl{n^c}}\Lambda(d)
=\sum_{d\le x^c}\Lambda(d)\sum_{\substack{n\le x\\d\,\mid\,\fl{n^c}}}1
=\sum_{p\le x^c}\log p\sum_{\substack{n\le x\\p\,\mid\,\fl{n^c}}}1
+E
\end{equation*}
where
\begin{align*}
0\le E\le\sum_{k\ge 2,~p\le x^{c/k}}\log p
\sum_{\substack{n\le x\\ p^{2\fl{k/2}}\,\mid\,\fl{n^c}}}1
=\sum_{d\le x^c}a_d\sum_{\substack{n\le x\\ d^2\,\mid\,\fl{n^c}}}1.
\end{align*}
Here,
$$
a_d=\sum_{\substack{k\ge 2,~p\le x^{c/k}\\ p^{\fl{k/2}}=d}}\log p\le 2\log d
\qquad(d\le x^c).
$$
By Proposition~\ref{(dd)prop:1} we have $E\ll x$, and
Theorem~\ref{(intro)thm:5} follows immediately.
\end{proof}

\section{Large prime factors of $\fl{n^c}$}
\label{sec:largeprimefactor}

\begin{proof}[Proof of Theorem~\ref{(intro)thm:1}
for $c\in(\tfrac{24979}{20803},\tfrac{5}{3})$.]
Let $\delta=\eps^2$.  We  show that
\begin{equation}
\label{(lpf)eq:1}
\sum_{p\le x^{\theta(c)-\delta}}\log p\sum_{\substack{n\le x\\p\,\mid\,\fl{n^c}}}1
\le(\theta(c)+O(\eps))\,x\log x
\end{equation}
for all large $x$.  In conjunction with Theorem~\ref{(intro)thm:5} this establishes
that there is a positive proportion of 
natural numbers $n\le x$ divisible by some prime $p\ge x^{\theta(c)-\delta}$;
thus, $P(n)>n^{\theta(c)-\delta}$ for such $n$.

We cover $[1,x^{\theta(c)-\eps}]$ with $O(\log x)$ abutting intervals of the form
$$
\cI_D=[D,(1+\eps)D]
$$
with $1\le D\le x^{\theta(c)-\eps}$.  For each $D$ we cover $[1,x^c/D]$ with
$O(\log x)$ abutting intervals of the form
$$
\cJ_L=[L,(1+\eps)L]
$$
with $1\le L\le x^c/D$.  As in the proof of Lemma~\ref{(notate)lem:2},
the double sum in~\eqref{(lpf)eq:1} is
\begin{equation}
\label{(lpf)eq:2}
\sum_{p\le x^{\theta(c)-\eps}}\log p
\sum_{\ell\le x^c/p}\bigl(\fl{-(p\ell)^\gamma}-\fl{-(p\ell+1)^\gamma}\bigr)
+O(x^{\theta(c)-\eps}).
\end{equation}
Arguing as we did after~\eqref{(dd)eq:3}, the contribution to~\eqref{(lpf)eq:2}
from the pairs $(p,\ell)$ that lie in $\cI_D\times \cJ_L$ is at most
$$
W_{D,L}(\log D)(DL)^{\gamma-1}(\gamma+O(\eps))+O\(\frac{W_{D,L}}{H_1}
+\sum_{h\le H_1}\frac{1}{h}\left|\sum_{(p,\ell)\in\cI_D\times\cJ_L}
\e(h(p\ell)^\gamma)\right|\),
$$
where
$$
H_1=DLx^{-1+\delta}\mand
W_{D,L}=\#\big\{(p,\ell)\in\cI_D\times\cJ_L\big\}.
$$
Now
\begin{equation*}
\begin{split}
\sum_{D,L}W_{D,L}(\log D)(DL)^{\gamma-1}(\gamma+O(\eps))
&\le (1+O(\eps))\sum_{p\le x^{\theta(c)-\delta}}\log p\sum_{\ell\le x^c/p}\gamma(p\ell)^{\gamma-1}\\
&\le (1+O(\eps))\,x\sum_{p\le x^{\theta(c)-\delta}}\frac{\log p}{p}\\
&\le (\theta(c)+O(\eps))\,x\log x.
\end{split}
\end{equation*}
Hence it suffices to show that for any pair $(D,L)$,
any number $H\in[1,H_1]$, and any sequence
$(a_h)_{h\sim H}$  of complex numbers with $|a_h|\le 1$,
the following bound holds uniformly:
$$
S^*=\sum_{h\sim H}a_h\sum_{(p,\ell)\in\cI_D\times\cJ_L}
\e(h(p\ell)^\gamma)\ll Hx^{1-\delta}.
$$
We consider three separate cases.

{\underline{\sc Case 1}}: $c\in[\tfrac{243}{205},\tfrac{112}{87})$.  We use
\cite[Theorem~3]{RoSa} to obtain the bound
\begin{equation}
\label{(lpf)eq:3}
S^*\ll x^\delta DLH\(\(\frac{F}{DL^2H}\)^{1/4}+L^{-1/2}+F^{-1}\).
\end{equation}
Here we write
$$
F=H(DL)^\gamma\le Hx.
$$
The last two terms in~\eqref{(lpf)eq:3} are handled easily, for
$$
x^\delta DL^{1/2}H\ll x^{c/2+\delta}D^{1/2}H\ll Hx^{1-\delta}
$$
since $D\ll x^{2-c-4\delta}$, whereas
$$
x^\delta DLHF^{-1}=x^\delta (DL)^{1-\gamma}\ll x^{c-1+\delta}\ll Hx^{1-\delta}.
$$
For the first summand, we have
\begin{align*}
x^{\delta} DLH\(\frac{F}{DL^2H}\)^{1/4}
&=x^{\delta} (DL)^{1/2}D^{1/4}H^{3/4}F^{1/4}\\
&\le x^{\delta} (x^c)^{1/2}(x^{3-2c-\eps})^{1/4}H^{3/4}(Hx)^{1/4}
\ll Hx^{1-\delta}
\end{align*}
since $D\le x^{3-2c-\eps}$.  This completes the proof of in Case~1.

Now suppose $c\ge\frac{112}{87}$.  Before separating the argument further,
we observe that (using the Kusmin-Landau inequality as in the proof
of Proposition~\ref{(dd)prop:1}) it suffices to consider the case that
$F\ge\eta L$ for an appropriate constant $\eta\asymp 1$.
Following the argument that gives \cite[(6.10)]{RoSa} we have
\begin{eqnarray}
&&S^*\ll\frac{L}{F^{1/2}}\int_{-1/2}^{1/2}
\Biggl|\,\sum_{h\sim H}\sum_{d\sim D}\sum_{V<\nu\le V_1}\e(\nu t)\,
\e\(\frac{Y h^{\overline\beta} d^{\overline\alpha}
\nu^{\overline\alpha}}{H^{\overline\beta} D^{\overline\alpha} V^{\overline\alpha}}\)
\Biggl|\,\min\big\{L,|t|^{-1}\big\}\,dt\nonumber\\
\label{(lpf)eq:4}
&&\hphantom{S^*}\quad+DLHF^{-1/2}+DH\log D,
\end{eqnarray}
where
$$
V\asymp V_1\asymp F/L,\qquad
Y\asymp F,\qquad
\overline\beta=\frac{1}{1-\gamma},\qquad
\overline\alpha=\frac{\gamma}{1-\gamma}.
$$
Since $F\gg L$ it is clear that
\begin{align*}
DLHF^{-1/2}+DH\log D&\ll DL^{1/2}Hx^{\delta}
\le D^{1/2}Hx^{c/2+\delta}\\
&\le Hx^{(\theta(c)+c)/2+\delta}
\ll Hx^{1-\delta},
\end{align*}
thus it remains only to bound the integral in~\eqref{(lpf)eq:4}.
We group the variables $h,d,\nu$ differently in the
next two cases.

{\underline{\sc Case 2}}: $c\in[\frac{112}{87},\frac{160}{117})$.
To bound the integrand, we apply Theorem~\ref{(intro)thm:3} pointwise
with $(M,M_1,M_2)$ replaced by $(D,H,F/L)$, and thus
it suffices to verify that
$$
D^{5/8}N^{7/8}F^{1/8}+\cdots+D^{11/10}NF^{-1/4}
\ll (F^{1/2}L^{-1})Hx^{1-2\delta}.
$$
Since $F=H(DL)^\gamma$ and $N=M_1M_2=H^2(DL)^\gamma L^{-1}$, 
this gives rise to nine upper bounds of the form
\begin{equation}
\label{(lpf)eq:5}
D^rL^sH^t(DL)^{\gamma u}\ll x^{v-C\delta},
\end{equation}
where $C$ is a positive constant (not necessarily the same at each
occurrence) and the numbers $r,s,t,u,v$ satisfy
$$
t\ge 0,\qquad s+t\ge 0,\qquad u\ge 0,\qquad r\ge s.
$$
Using the inequalities $H\le DLx^{-1+\delta}$ and $DL\le x^c$,
we see that the left-hand side of~\eqref{(lpf)eq:5} is
$$
\le D^{r+t}L^{s+t}x^{-t+u+t\delta}
= D^{r-s}(DL)^{s+t}x^{-t+u+t\delta}
\le D^{r-s}x^{c(s+t)-t+u+t\delta};
$$
therefore,~\eqref{(lpf)eq:5} holds provided that
\begin{equation}
\label{(lpf)eq:6}
D\le x^{(v+t-u-c(s+t))/(r-s)-\eps}.
\end{equation}
Taking all nine bounds into account, we must have
$D\le x^{\theta_1(c)-\eps}$, where
$$
\theta_1(c)=\min\big\{\tfrac{7-4c}{4},\tfrac{7-3c}{7},\tfrac{92-49c}{68},
\tfrac{54-28c}{42},\tfrac{54-29c}{39},\tfrac{266-139c}{192},\tfrac{100-53c}{74},
\tfrac{6-3c}{5},\tfrac{20-5c}{22}\big\}.
$$
After a simple computation one verifies that
$$
\theta_1(c)=\tfrac{92-49c}{68}=\theta(c)\qquad\text{for all}
\quad c\in[\tfrac{112}{87},\tfrac{160}{117}),
$$
so this completes the proof in Case~2.

{\underline{\sc Case 3}}: $c\in[\frac{160}{117},\frac{5}{3})$.  We proceed
just as in Case~2 but with the roles of $D$ and $H$ interchanged, i.e.,
we apply Theorem~\ref{(intro)thm:3} pointwise with $(M,M_1,M_2)$ replaced
by $(H,D,F/L)$, and we have $N=M_1M_2=DH(DL)^\gamma L^{-1}$.  We obtain
nine new bounds of the form~\eqref{(lpf)eq:6} with different values
of $r,s,t,u,v$, and this leads to the requirement that
$D\le x^{\theta_2(c)-\eps}$, where
$$
\theta_2(c)=\min\left\{\tfrac{5-2c}{6},\tfrac{8-4c}{6},\tfrac{74-31c}{86},\tfrac{46-20c}{50},
\tfrac{43-18c}{50},\tfrac{230-103c}{228},\tfrac{82-35c}{92},\tfrac{22-7c}{20}\right\}.
$$
After a calculation, one verifies that
$\theta_2(c)=\theta(c)$ for all $c\in[\tfrac{160}{117},\tfrac53)$.
This completes the proof in Case~3 and finishes the proof of
Theorem~\ref{(intro)thm:1} for values of $c$ in the interval
$[\tfrac{24979}{20803},\tfrac{5}{3})$.
\end{proof}

Not far to the right of $c=\tfrac85$, it becomes more efficient to estimate the exponential sum
$$
\sum_{n\sim N}\e\(\frac{hn^c}{q}\)
$$
in order to give a good lower bound for $P(\fl{n^c})$.
We use this approach for values of $c\ge\tfrac53$.

\begin{proposition}
\label{(lpf)prop:2}
$(a)$ Fix $c\in(\frac32,2)$.  For any natural number $q\le N^{(3-c)/6-3\eps}$
and any integer $a$ we have
\begin{equation}
\label{(lpf)eq:7}
\#\big\{n\sim N:\fl{n^c}\equiv a\pmod q\big\}
=\frac{N}{q}+O\(\frac{N^{1-\eps}}{q}\).
\end{equation}
$(b)$ There exists a constant $\beta>0$ with the property that for any fixed
$c>2$, $c\not\in\Z$, the estimate~\eqref{(lpf)eq:7} holds for all $q\le N^{\beta/c^2}$
and $a\in\Z$.
\end{proposition}

From Proposition~\ref{(lpf)prop:2} we derive the following
corollary, which establishes Theorem~\ref{(intro)thm:1}
for any $c\in[\tfrac{5}{3},2)$ and also establishes Theorem~\ref{(intro)thm:2}.

\begin{corollary}
\label{(lpf)cor:1}
Let
$$
\theta_3(c)=\begin{cases}
(3-c)/6&\quad\text{if $\tfrac53\le c<2$};\\
\beta/c^2&\quad\text{if $c>2$, $c\not\in\Z$}.
\end{cases}
$$
Then
\begin{equation}
\label{(lpf)eq:8}
P(\fl{n^c})>n^{\theta_3(c)-\eps}
\end{equation}
for infinitely many $n$.
\end{corollary}

\begin{proof}
Let $p$ be a prime in the interval $[\frac12 N^{\theta_3(c)-\eps/2},N^{\theta_3(c)-\eps/2}]$.
Applying Proposition~\ref{(lpf)prop:2} with $\eps/6$ in place of $\eps$,
the number of $n\sim N$ for which $p\mid\fl{n^c}$ is
$\gg N/p \gg N^{1-\theta_3(c)+\eps/2}$ for all large $N$,
and~\eqref{(lpf)eq:8} holds for every such $n$.
\end{proof}

\begin{lemma}
\label{(lpf)lem:1}
There is a constant $b\in(0,1)$ such that for any $c>2$, $c\not\in\Z$, the bound
$$
\sum_{n\sim N}\e(\alpha n^c)\ll N^{1-b/c^2}
$$
holds uniformly for all $\alpha$ such that 
$N^{-c/2}\le|\alpha|\le N^{c/2}$,
where the implied constant depends only on $c$.
\end{lemma}

\begin{proof}
This is a special case of Karatsuba~\cite[Theorem~1]{Kara}; see
also Br\"udern and Perelli \cite[Lemma~10]{BruPer}.  One can adapt the
work of Baker and Kolesnik~\cite{BakKol} to give an explicit
value for $b$; an even larger value for $b$ would follow by
incorporating the recent work of Wooley \cite{Wooley}.
\end{proof}

\begin{proof}[Proof of Proposition~\ref{(lpf)prop:2}]
The condition $\fl{n^c}\equiv a\pmod q$ is equivalent to
\begin{equation}
\label{(lpf)eq:9}
\frac{a}{q}\le\left\{\frac{n^c}{q}\right\}<\frac{a+1}{q}.
\end{equation}
According to Lemma~\ref{(notate)lem:3}, the number of $n\sim N$
for which~\eqref{(lpf)eq:9} holds is
$$
\frac{N}{q}+O\(\frac{N^{1-\eps}}{q}\)+
O\(\,\sum_{1\le h\le qN^\eps}\left|\,\sum_{n\le N}\e\(\frac{hn^c}{q}\)\right|\,\).
$$
Thus, to deduce $(a)$ it suffices show that the bound
\begin{equation}
\label{(lpf)eq:10}
\sum_{n\le N}\e\(\frac{hn^c}{q}\)\ll\frac{N^{1-2\eps}}{q}\qquad(1\le h\le qN^\eps)
\end{equation}
holds for any $q\le N^{(3-c)/6-3\eps}$.
We apply Lemma~\ref{(notate)lem:1}$\,(ii)$ with $\lambda\asymp hN^{c-3}q^{-1}$,
which gives
\begin{align*}
\sum_{n\le N}\e\(\frac{hn^c}{q}\)
&\ll \frac{N}{q}\(h^{1/6}q^{5/6}N^{(c-3)/6}+qN^{-1/4}
+h^{-1/4}q^{5/4}N^{-c/4}\).
\end{align*}
Taking into account the following bounds,
which are valid for any $c\in(\tfrac32,3)$:
\begin{align*}
&h^{1/6}q^{5/6}N^{(c-3)/6}\le qN^{(c-3)/6+\eps}
\le  N^{-2\eps},\\
&qN^{-1/4}\le N^{1/4-c/6-3\eps}\le N^{-2\eps},\\
&h^{-1/4}q^{5/4}N^{-c/4}
\le q^{5/4}N^{-c/4}\le N^{5/8-11c/24-2\eps}\le N^{-2\eps},
\end{align*}
we finish the proof of $(a)$.

For part~$(b)$, choose any positive $\beta<\min\{1,b\}$, where $b$ is the
constant of Lemma~\ref{(lpf)lem:1}.  We must prove
\eqref{(lpf)eq:10} for any $q\le N^{\beta/c^2}$.  Clearly, if $\eps>0$
is sufficiently small we have
$$
N^{-c/2}\le N^{-\beta/c^2}\le\frac{h}{q}\le N^\eps\le N^{c/2},
$$
and by Lemma~\ref{(lpf)lem:1} it follows that
$$
\sum_{n\le N}\e\(\frac{hn^c}{q}\)\ll N^{1-b/c^2}\ll \frac{N^{1-2\eps}}{q}\,.
$$
and this completes the proof of $(b)$.
\end{proof}

\section{Smooth values of  $\fl{n^c}$}
\label{sec:sv}

The proof of Theorem~\ref{(intro)thm:6} is based 
on the following result which we prove by adapting Heath-Brown~\cite{HB82}.

\begin{proposition}
\label{(sv)prop:1}
Fix $c\in(1,\tfrac{24979}{20803})$.
Let $(a_k)_{k\in\N}$ be a bounded sequence of non-negative
numbers for which
\begin{equation}
\label{(sv)eq:1}
\sum_{k\sim K}a_k\gg\frac{K}{\log K}
\end{equation}
for all large $K\le\tfrac12 x$.  Put
$$
K=x^{c-1+6\eps},\qquad
L=\tfrac15x^{1-6\eps}\mand
R(n)=\sum_{\substack{(k,\ell)\sim(K,L)\\k\ell=\fl{n^c}}}a_ka_\ell.
$$
Then
$$
\sum_{n\le x}R(n)\gg x^{1-\eps}.
$$
\end{proposition}

\begin{proof} 
In view of Lemma~\ref{(notate)lem:2} we have
$$
\sum_{n\le x}R(n)=T_0+T_1+O(1),
$$
where
$$
T_0=\gamma\sum_{(k,\ell)\sim(K,L)}a_ka_\ell(k\ell)^{\gamma-1}
\gg(KL)^{\gamma-\eps}\gg x^{1-\eps}
$$
from~\eqref{(sv)eq:1}, whereas
$$
T_1=\sum_{(k,\ell)\sim(K,L)}a_ka_\ell
\bigl(\psi(-(k\ell+1)^\gamma)-\psi(-(k\ell)^\gamma)\bigr).
$$
Hence, it suffices to show that $T_1\ll x^{1-2\eps}$.

Using~\eqref{(notate)eq:3} and writing $\psi^*(t)=\sum_{0<|h|\le H}c_h\,\e(th)$,
$y_{k\ell}=-(k\ell+1)^\gamma$, $z_{k\ell}=-(k\ell)^\gamma$,
we see that $T_1\ll S_1+S_2+S_3$, where
\begin{align*}
S_1&=\biggl|\sum_{(k,\ell)\sim(K,L)}a_ka_\ell\bigl(\psi^*(y_{k\ell})-\psi^*(z_{k\ell})\bigr)\biggl|,\\
S_2&=\sum_{|h|\le H}d_h\sum_{(k,\ell)\sim(K,L)}\e(hy_{k\ell}),
\end{align*}
and $S_3$ is defined as $S_2$ with $z_{k\ell}$ instead of $y_{k\ell}$.
We choose $H=x^{c-1+\eps}$, so that the contribution to $S_2+S_3$ from $h=0$ is
$O\(KLH^{-1}\)=O\(x^{1-\eps}\)$.

To bound the contribution to $S_2+S_3$ for nonzero $h$, we use the
exponent pair $(\tfrac12,\tfrac12)$ for the sum over $\ell$
and treat the sums over $k,h$ trivially.  For example, 
$$
\left|\frac{d}{dt}\bigl(h(kt+1)^\gamma\bigr)\right|\asymp
|h| (x^c)^{\gamma-1}K=|h|x^{6\eps}.
$$
Since  $x^{6\eps}\ll |h|x^{6\eps} \ll x^{c-1+7\eps}$  
for any $c<2$ we have
$$
\sum_{k\sim K}\left| \sum_{\ell\sim L} \e(hy_{k\ell})\right|\ll K L^{1/2}(x^{c-1+7\eps})^{1/2}
\ll x^{3c/2 -1 + 7\eps}\ll x^{1-\eps}.
$$

The sum $S_1$ is treated using a partial summation argument given
in Heath-Brown~\cite{HB83} with $R(n)$ replacing $\Lambda(n)$.  It suffices
to show that
$$
\sum_{h\le H}\eps_h\sum_{B<n\le B_1}R(n)\,\e(hn^\gamma)\ll Bx^{-\eps},
$$
where $B=KL$, $B_1$ is an arbitrary number in $(B,4B]$,
and $|\eps_h|=1$ for each~$h$.  We can rewrite this as
$$
\sum_{h\le H}\eps_h\sum_{k\sim K}a_k\sum_{B/k<\ell\le B_1/k}
a_\ell\,\e(h(k\ell)^\gamma)\ll Bx^{-\eps}
$$
By a standard technique (explained, e.g., in Harman~\cite[\S3.2]{Har2007})
we need only show that the bound
$$
S=\sum_{h\sim H'}\eps_h\sum_{k\sim K}b_k\sum_{\ell\sim L}
c_\ell\,\e(h(k\ell)^\gamma)\ll KLx^{-2\eps}
$$
holds whenever $H'\le H$, $|b_k|\le 1$, $|c_\ell|\le 1$.  We use
Baker~\cite[Theorem~2]{Ba1994}.  It is easy to check
the hypothesis $X\gg L_1L_2$ holds with the choice $X=H'(KL)^\gamma$,
$L_1=H'$, and $L_2=K$; hence, for any exponent pair $(\kappa,\lambda)$
we derive that
$$
S\ll\bigl((H'K)^{1/2}L+
(HK)^{\frac{2+\kappa}{2+2\kappa}}(H'K^\gamma L^\gamma)^{\frac{\kappa}{2+2\kappa}}
L^{\frac{1+\kappa+\lambda}{2+2\kappa}}\bigr)\log x.
$$
Examining the `worst' case in the proof of \cite[Theorem~2]{Ba1994}
leads us to choose the exponent pair (see~\cite{Huxley})
$$
(\kappa,\lambda)=BA^4(\tfrac{32}{205}+\eps,\tfrac12+\tfrac{32}{205}+\eps)
=(\tfrac{3843}{8480},\tfrac{4304}{8480})+O(\eps).
$$
Noting that the bound
$$
(HK)^{1/2}L\log x\ll KL x^{-2\eps}
$$
follows from the identity $H=Kx^{-5\eps}$, it remains to show that
$$
HK^{\frac{2+\kappa}{2+2\kappa}}(K^\gamma L^\gamma)^{\frac{\kappa}{2+2\kappa}}
L^{\frac{1+\kappa+\lambda}{2+2\kappa}}\log x\ll KLx^{-2\eps}.
$$
Recalling our choices of $K$, $L$ and $H$,
we are led to the bound
$$
(2+\kappa)(c-1)<1-\lambda,\qquad\text{or}\qquad
c<\tfrac{24979}{20803}.
$$
This completes the proof.
\end{proof}

Proposition~\ref{(sv)prop:1} immediately yields the following result.

\begin{corollary}
\label{(sv)cor:1}
For any fixed $c\in(1,\tfrac{24979}{20803})$ we have
\begin{itemize}
\item[$(a)$] For at least $C_0x^{1-\eps}$ natural numbers $n\le x$ one has $P(\fl{n^c})\le n^\eps$;
\item[$(b)$] For at least $C_0x^{1-\eps}$ natural numbers $n\le x$ one has $P(\fl{n^c})\ge n^{2-c-\eps}$;
\end{itemize}
where $C_0>0$ depends only $c$ and $\eps$. 
\end{corollary}

The reader can easily obtain Corollary~\ref{(sv)cor:1} by taking
$(a_k)_{k\in\N}$ to be the indicator function either of the integers
with $P(k)\le x^{\eps/2}$, or of the prime numbers.  Note that assertion $(b)$
completes the proof of Theorem~\ref{(intro)thm:1}
for values of $c$ in the interval $[\tfrac{243}{205},\tfrac{24979}{20803})$.

\section{Carmichael numbers composed of Piatets\-ki-Shapiro primes}
\label{sec:carm}

Our first goal is to establish two preliminary lemmas that are
needed for an application of Lemma~\ref{(notate)lem:4} with
the function
$$
f(x)=\e(mx^\gamma+xh/d),
$$
where $m,h,d\in\N$.  In what follows,
we suppose that $1<N<N_1\le 2N$.

\begin{lemma}
\label{(carm)lem:apple} Suppose $|a_k|\le 1$ for all $k\sim K$.
Fix $\gamma\in(0,1)$ and $m,h,d\in\N$.  Then, for any $L\gg N^{2/3}$
the Type~I sum
$$
S_I=\mathop{\sum_{k\sim K}\sum_{\ell\sim L}}\limits_{N<k\ell\le N_1}
a_k\,\e\bigl(mk^\gamma\ell^\gamma+k\ell h/d\bigr)
$$
satisfies the bound
$$
S_I\ll m^{1/2}N^{1/3+\gamma/2}+m^{-1/2}N^{1-\gamma/2}.
$$
\end{lemma}

\begin{proof}
Writing $F(\ell)=mk^\gamma\ell^\gamma+k\ell h/d$
we see that
$$
|F''(\ell)|=m\gamma(1-\gamma)k^\gamma\ell^{\gamma-2}
\asymp mK^\gamma L^{\gamma-2}\qquad(\ell\sim L).
$$
Using Lemma~\ref{(notate)lem:1} it follows that
$$
\sum_{\substack{\ell\sim L\\N<k\ell\le N_1}}\e(mk^\gamma \ell^\gamma+k\ell h/d)
\ll m^{1/2}K^{\gamma/2}L^{\gamma/2}
+m^{-1/2}K^{-\gamma/2}L^{1-\gamma/2}.
$$
Since $|a_k|\le 1$ for all $k\sim K$ we see that
\begin{equation*}
\begin{split}
S_I
&\le\sum_{k\sim K}\biggl|
\sum_{\substack{\ell\sim L\\N<k\ell\le N_1}}
\e(mk^\gamma \ell^\gamma+k\ell h/d)\biggl|\\
&\ll
m^{1/2}K^{1+\gamma/2}L^{\gamma/2}
+m^{-1/2}K^{1-\gamma/2}L^{1-\gamma/2}.
\end{split}
\end{equation*}
Noting that $KL\asymp N$ (else the result is trivial)
and so $K\ll N^{1/3}$, we finish the proof.
\end{proof}

\begin{lemma}
\label{(carm)lem:orange}
Suppose $|a_k|\le 1$ and $|b_\ell|\le 1$ for $(k,\ell)\sim(K,L)$.
Fix $\gamma\in(0,1)$ and $m,h,d\in\N$.  Then, for any $K$ in the range
$N^{1/3}\ll K\ll N^{1/2}$ the Type~II sum
$$
S_{I\!I}=
\mathop{\sum_{k\sim K}~\sum_{\ell\sim L}}\limits_{N<k\ell\le N_1}
a_k\,b_\ell\,\,\e(mk^\gamma\ell^\gamma+k\ell h/d)
$$
satisfies the bound
$$
S_{I\!I}\ll m^{-1/4}N^{1-\gamma/4}
+m^{1/6}N^{7/9+\gamma/6}+N^{11/12}.
$$
\end{lemma}

\begin{proof}
We can assume that $KL\asymp N$.
By Lemma~\ref{(notate)lem:5} we have
\begin{equation}
\label{(carm)eq:utah}
|S_{I\!I}|^2\ll K^2L^2Q^{-1}+KLQ^{-1}\sum_{\ell\sim L}
\sum_{0<|q|\le Q}|S(q;\ell)|,
\end{equation}
where
$$
S(q;n)=\sum_{k\in I(q;\ell)}\e(F(k)),\qquad
F(k)=mk^\gamma(\ell^\gamma-(\ell+q)^\gamma)-kqh/d,
$$
and each $I(q;n)$ is a certain subinterval in the set of numbers $k\sim K$.
Since
$$
|F''(k)|=m\gamma(1-\gamma)k^{\gamma-2}((\ell+q)^\gamma-\ell^\gamma)
\asymp mK^{\gamma-2}L^{\gamma-1}q\qquad(k\sim K),
$$
it follows from Lemma~\ref{(notate)lem:1} that
$$
S(q;\ell)\ll K(mK^{\gamma-2}L^{\gamma-1}q)^{1/2}
+(mK^{\gamma-2}L^{\gamma-1}q)^{-1/2}.
$$
Inserting this bound in~\eqref{(carm)eq:utah}  and
summing over $\ell$ and $q$, we derive that
\begin{align*}
|S_{I\!I}|^2
&\ll K^2L^2Q^{-1}+m^{1/2}
K^{1+\gamma/2}L^{3/2+\gamma/2}Q^{1/2}
+m^{-1/2}K^{2-\gamma/2}L^{5/2-\gamma/2}Q^{-1/2}\\
&\qquad\ll N^2Q^{-1}+m^{1/2}
K^{-1/2}N^{3/2+\gamma/2}Q^{1/2}
+m^{-1/2}K^{-1/2}
N^{5/2-\gamma/2}Q^{-1/2},
\end{align*}
where we used the fact that $KL\asymp N$ in the second step.
Since the above holds whenever $0<Q\le L$,
an application of Lemma~\ref{(esm)lem:4} gives
$$
|S_{I\!I}|^2 \ll KN
+m^{-1/2}N^{2-\gamma/2}
+m^{1/3}K^{-1/3}N^{5/3+\gamma/3} +K^{-1/2}N^2.
$$
Finally, for $K$ in the range $N^{1/3}\ll K\ll N^{1/2}$ we arrive
at the bound
$$
|S_{I\!I}|^2\ll m^{-1/2}N^{2-\gamma/2}
+m^{1/3}N^{14/9+\gamma/3}+N^{11/6},
$$
and the result follows.
\end{proof}

For any coprime integers $a$ and $d\ge 1$,
we denote by $\sP_{d,a}^{(c)}$ the set of Piatetski-Shapiro
primes in the arithmetic progression $a$~mod~$d$; that is,
$$
\sP_{d,a}^{(c)}=\big\{p\equiv a\bmod d:p=\fl{n^c}\text{~for some~}n\in\N\big\}.
$$
Our next goal is to estimate the counting functions
$$
\pi_c(x;d,a)=\#\big\{p\le x:p\in\sP_{d,a}^{(c)}\big\}
\mand
\vartheta_c(x;d,a)=\sum_{\substack{p\le x\\p\in\sP_{d,a}^{(c)}}}\log p
$$
in terms of the more familiar functions
$$
\pi(x;d,a)=\#\big\{p\le x:p\equiv a\bmod d\big\}
\mand
\vartheta(x;d,a)=\hskip-5pt\sum_{\substack{p\le x\\p\equiv a\bmod d}}
\hskip-5pt\log p.
$$

By Lemma~\ref{(notate)lem:2} we have
$$
\pi_c(x;d,a)=\Sigma_1(x)+\Sigma_2(x)+O(1),
$$
where
\begin{equation*}
\begin{split}
\Sigma_1(x)&=\gamma\hskip-5pt\sum_{\substack{p\le x\\p\equiv a\bmod d}}\hskip-5pt
p^{\gamma-1},\\
\Sigma_2(x)&=\hskip-5pt\sum_{\substack{p\le x\\p\equiv a\bmod d}}\hskip-5pt
\bigl(\psi(-(p+1)^\gamma)-\psi(-p^\gamma)\bigr).
\end{split}
\end{equation*}
Using partial summation one sees that
$$
\Sigma_1(x)
=\gamma x^{\gamma-1}\,\pi(x;d,a)
-\gamma(\gamma-1)\int_2^x u^{\gamma-2}\,\pi(u;d,a)\,du.
$$
Next, we turn our attention to $\Sigma_2(x)$.
We begin by considering sums of the form
\begin{equation}
\label{(carm)eq:Ssum}
S=\sum_{\substack{N<n\le N_1\\n\equiv a\bmod d}}\Lambda(n)
\bigl(\psi(-(n+1)^\gamma)-\psi(-n^\gamma)\bigr).
\end{equation}
Arguing as in~\cite[pp.~47--49]{GraKol},
for any real number $M\ge 1$ we derive the uniform bound
\begin{equation}
\label{(carm)eq:rainday}
S\ll N^{\gamma-1}\max_{N_2\sim N}\sum_{1\le m\le M}
\left|\,\sum_{\substack{N<n\le N_2\\n\equiv a\bmod d}}
\Lambda(n) \e(mn^\gamma)\right|
+NM^{-1}+N^{\gamma/2}M^{1/2}.
\end{equation}
To bound the inner sum, we note that
$$
\sum_{\substack{N<n\le N_2\\n\equiv a\bmod d}}
\Lambda(n) \e(mn^\gamma)
=\frac1d\sum_{h=1}^d\sum_{N<n\le N_2}\Lambda(n) \e(mn^\gamma+(n-a)h/d),
$$
hence it suffices to give a bound on exponential sums of the form
$$
T=\sum_{N<n\le N_2}\Lambda(n) \e(mn^\gamma+nh/d),
$$
where $1<N<N_2\le 2N$. We do this with an application of
Lemma~\ref{(notate)lem:4}, taking into account
the estimates of Lemmas~\ref{(carm)lem:apple} and~\ref{(carm)lem:orange};
we find that
\begin{align*}
TN^{-\eps}&\ll
m^{1/2}N^{1/3+\gamma/2}
+m^{1/6}N^{7/9+\gamma/6}
+m^{-1/4}N^{1-\gamma/4}
+N^{11/12}
\end{align*}
for any fixed $\eps>0$.  Inserting this bound in
\eqref{(carm)eq:rainday} and summing over $m$, it follows that
\begin{align*}
SN^{-\eps}&\ll
N^{-2/3+3\gamma/2}M^{3/2}
+N^{-2/9+7\gamma/6}M^{7/6}\\
&\quad+N^{3\gamma/4}M^{3/4}
+N^{-1/12+\gamma}M
+NM^{-1}.
\end{align*}
Since the above holds for any real $M\ge 1$,
using Lemma~\ref{(esm)lem:4}  we find that
\begin{align*}
SN^{-\eps}
&\ll N^{-2/3+3\gamma/2}
+N^{-2/9+7\gamma/6}
+N^{3\gamma/4}
+N^{-1/12+\gamma}\\
&\quad+N^{1/3+3\gamma/5}
+N^{17/39+7\gamma/13}
+N^{3/7+3\gamma/7}
+N^{11/24+\gamma/2}.
\end{align*}
Since $\pi_c(x;d,a)\ll x^\gamma$, this bound is trivial unless
the exponent of each term in the parentheses is strictly less
than $\gamma$. Thus, from now on we assume that
$\gamma\in\(\tfrac{17}{18},1\)$.  In this case, after eliminating lower order
terms, the previous bound simplifies to
\begin{equation}
\label{(carm)eq:bndSNN1}
S\ll N^{17/39+7\gamma/13+\eps}
\end{equation}
for any fixed $\eps>0$.

To bound $\Sigma_2(x)$, let
\begin{equation*}
\begin{split}
G(x)&=\sum_{\substack{p\le x\\p\equiv a\bmod d}}(\log p)
\bigl(\psi(-(p+1)^\gamma)-\psi(-p^\gamma)\bigr),\\
H(x)&=\sum_{\substack{n\le x\\n\equiv a\bmod d}}\Lambda(n)
\bigl(\psi(-(n+1)^\gamma)-\psi(-n^\gamma)\bigr).
\end{split}
\end{equation*}
Clearly,
$$
H(x)=G(x)+O(x^{1/2}),
$$
and by partial summation,
$$
\Sigma_2(x)=\frac{G(x)}{\log x}+\int_2^x\frac{G(u)}{u(\log u)^2}\,du.
$$
Splitting the sum $H(x)$ into $O(\log x)$ sums $S$ of the
form~\eqref{(carm)eq:Ssum}
with $2N\le x$, and using~\eqref{(carm)eq:bndSNN1},
we see that the bound $H(x)\ll x^{17/39+7\gamma/13+\eps}$
holds for any fixed $\eps>0$, and from the preceding observations
we derive a similar result for $\Sigma_2(x)$. Putting everything
together, we have proved Theorem~\ref{(intro)thm:9}.

Replacing the function $\pi_c(x;d,a)$ with the weighted counting
function
$$
\vartheta_c(x;d,a)=\sum_{\substack{p\le x\\p\in\sP_{d,a}^{(c)}}}\log p
=\hskip-5pt\sum_{\substack{p\le x\\p\equiv a\bmod d}}
\hskip-5pt\bigl(\fl{-p^\gamma}-\fl{-(p+1)^\gamma}\bigr)\log p
$$
and using a similar argument, we obtain the following statement.

\begin{theorem}
\label{(carm)thm:pinkfloyd}
For any $c\in\(1,\tfrac{18}{17}\)$ and $\eps>0$ we have
\begin{align*}
\vartheta_c(x;d,a)&=\gamma x^{\gamma-1}\,\vartheta(x;d,a)
+\gamma(1-\gamma)\int_2^x u^{\gamma-2}\,\vartheta(u;d,a)\,du\\
&\quad+O\bigl(x^{17/39+7\gamma/13+\eps}\bigr),
\end{align*}
where the implied constant depends only on $c,\eps$.
\end{theorem}

For the proof of Theorem~\ref{(intro)thm:8} we also require
the following variant of the Brun-Titchmarsh bound for Piatetski-Shapiro
primes, which is a consequence of Theorem~\ref{(intro)thm:9}.

\begin{theorem}
\label{(carm)thm:heavymetal}
For any $c\in\(1,\tfrac{18}{17}\)$ and
$A\in\(0,-\tfrac{17}{39}+\tfrac{6\gamma}{13}\)$
there is a number $C=C(c,A)>0$ such that if $\gcd(a,d)=1$ and
$1\le d\le x^A$, then the following bound holds:
$$
\pi_c(x;d,a)\le \frac{C\,x^{\gamma}}{\varphi(d)\log x}\,.
$$
\end{theorem}

\begin{proof}
Let $\eps>0$ be chosen (depending only on $c,A$) so that
$$
\max\big\{2A\gamma,\tfrac{17}{39}+\tfrac{7\gamma}{13}+\eps\big\}\le\gamma-A-\eps.
$$
Then, by Theorem~\ref{(intro)thm:9} it follows that
\begin{equation}
\label{(carm)eq:rush}
\pi_c(x;d,a)\ll x^{\gamma-1}\,\pi(x;d,a)
+\int_{x^{2A}}^x u^{\gamma-2}\,\pi(u;d,a)\,du
+x^{\gamma-A-\eps},
\end{equation}
where the implied constant depends only on $c,A$.  Since
$$
x^{\gamma-A-\eps}\ll\frac{x^{\gamma-A}}{\log x}\le\frac{x^\gamma}{\varphi(d)\log x}
\qquad(1\le d\le x^A),
$$
the result follows by applying the Brun-Titchmarsh theorem
to the right side of~\eqref{(carm)eq:rush}.
\end{proof}

We now outline our proof of Theorem~\ref{(intro)thm:8}.
We are brief since our construction of Carmichael numbers 
composed of primes from $\sP^{(c)}$ closely follows the
construction of ``ordinary'' Carmichael numbers given by 
Alford, Granville and Pomerance \cite{AGP}.
Here, we discuss only the changes that are needed
to establish Theorem~\ref{(intro)thm:8}.

The idea behind our proof is
to show that the set $\sP^{(c)}$ is sufficiently well-distributed
over arithmetic progressions so that, following the method of~\cite{AGP},
the primes used to form Carmichael numbers can all be drawn from
$\sP^{(c)}$ rather than the set $\sP$ of all prime numbers.
For this, we apply the results derived earlier in this section.

The following statement plays a crucial role in our construction
analogous to that played by \cite[Theorem~2.1]{AGP}.

\begin{lemma}
\label{(carm)lem:cotton}
Fix $c\in\(1,\tfrac{18}{17}\)$ and $B\in\(0,-\tfrac{17}{39}+\tfrac{6\gamma}{13}\)$.
There exist numbers $\eta>0$, $x_0$ and $D$
such that for all $x\ge x_0$ there is a set $\cD(x)$ consisting
of at most $D$ integers such that
$$
\biggl|\,\vartheta_c(x;d,a)-\frac{x^\gamma}{\varphi(d)}\biggl|\,
\le \frac{x^\gamma}{2\,\varphi(d)}
$$
provided that
\begin{itemize}
\item[$(i)$] $d$ is not divisible by any element of $\cD(x)$;
\item[$(ii)$] $1\le d \le x^B$;
\item[$(iii)$] $\gcd(a,d)=1$.
\end{itemize}
Every number in $\cD(x)$ exceeds $\log x$,
and all, but at most one, exceeds $x^{\eta}$.
\end{lemma}

\begin{remark} In the statement and proof of Lemma~\ref{(carm)lem:cotton},
$\eta$, $x_1$, $D$ and $\cD(x)$ all depend on the
choice of $c$ and $B$, but this is suppressed from
the notation for the sake of clarity.
\end{remark}

\begin{proof}
For any such $B$ we have $2B<\tfrac{5}{12}$.  Applying
\cite[Theorem~2.1]{AGP} (with $2B$ instead of $B$) we see that
there exist numbers $\eta>0$, $x_1$ and $D$
such that for all $x\ge x_1$ there is a set $\cD(x)$ consisting
of at most $D$ integers such that
\begin{equation}
\label{(carm)eq:ledzep}
\biggl|\,\vartheta(y;d,a)-\frac{y}{\varphi(d)}\biggl|\,
\le \frac{y}{10\,\varphi(d)}\qquad(x^{1-B}\le y\le x)
\end{equation}
whenever $(i)$, $(ii)$ and $(iii)$ hold.  Furthermore,
every number in $\cD(x)$ exceeds $\log x$,
and all, but at most one, exceeds $x^{\eta}$.

Let $\eps>0$ be chosen (depending only on $c,B$) so that
$$
\tfrac{17}{39}+\tfrac{7\gamma}{13}+\eps\le\gamma-B-\eps,
$$
and suppose that $d$ and $a$ are integers such that $(i)$, $(ii)$
and $(iii)$ hold. Then, by Theorem~\ref{(carm)thm:pinkfloyd} it follows that
$$
\vartheta_c(x;d,a)=T_1+T_2+T_3+O(T_4),
$$
where
\begin{equation*}
\begin{split}
T_1&=\gamma x^{\gamma-1}\,\vartheta(x;d,a),\\
T_2&=\gamma(1-\gamma)\int_{x^{1-B}}^x u^{\gamma-2}\,\vartheta(u;d,a)\,du,\\
T_3&=\gamma(1-\gamma)\int_2^{x^{1-B}} u^{\gamma-2}\,\vartheta(u;d,a)\,du,\\
T_4&=x^{\gamma-B-\eps}.
\end{split}
\end{equation*}
By~\eqref{(carm)eq:ledzep} we have
$$
0.9\,\gamma\,\frac{x^\gamma}{\varphi(d)}\le
T_1\le 1.1\,\gamma\,\frac{x^\gamma}{\varphi(d)}
$$
and
$$
0.9\,(1-\gamma)\,\frac{x^\gamma}{\varphi(d)}
+O\(\,\frac{x^{\gamma(1-B)}}{\varphi(d)}\)\le
T_2\le 1.1\,(1-\gamma)\,\frac{x^\gamma}{\varphi(d)}
+O\(\,\frac{x^{\gamma(1-B)}}{\varphi(d)}\).
$$
Using the Brun-Titchmarsh bound $\vartheta(x;d,a)\ll x/\varphi(d)$
for $1\le d\le x^B$ we also see that
$$
T_3\ll \frac{x^{\gamma(1-B)}}{\varphi(d)}\,.
$$
Finally, we note that
$$
T_4\le\frac{x^{\gamma-\eps}}{\varphi(d)}
\qquad(1\le d\le x^B).
$$
Combining the above estimates, we deduce that the inequalities
$$
(0.9+o(1))\,\frac{x^\gamma}{\varphi(d)}\le
\vartheta_c(x;d,a)\le (1.1+o(1))\,\frac{x^\gamma}{\varphi(d)}
$$
hold as $x\to\infty$, and the result follows.
\end{proof}

As an application of Lemma~\ref{(carm)lem:cotton} we derive the following
statement, which extends \cite[Theorem~3.1]{AGP} to the setting
of Piatetski-Shapiro primes.

\begin{lemma}
\label{(carm)lem:cotton2}
Fix $c\in\(1,\tfrac{18}{17}\)$, and let $A,B,B_1$ be positive
real numbers such that $B_1<B<A<-\tfrac{17}{39}+\tfrac{6\gamma}{13}$.
Let $C=C(c,A)>0$ have the property described in Theorem~\ref{(carm)thm:heavymetal}.
There exists a number $x_2=x_2(c,A,B,B_1)$
such that if $x\ge x_2$ and $L$ is a squarefree integer not
divisible by any prime $q$ exceeding $x^{(A-B)/2}$ and for which
\begin{equation}
\label{(carm)eq:acdc}
\sum_{\text{\rm prime~}q\,\mid\,L}\frac1q\le \frac{1-A}{16C}\,,
\end{equation}
then there is a positive integer $k\le x^{1-B}$ with $\gcd(k,L)=1$
such that 
\begin{equation*}
\begin{split}
&\#\big\{d\mid L:dk+1\le x
\text{~and~}p=dk+1\text{~is a prime in~}\sP^{(c)}\big\}\\
&\qquad\qquad\qquad\qquad
\ge\frac{2^{-D-2}(x^{1-B+B_1})^{\gamma-1}}{\log x}
\,\#\big\{d\mid L:x^{B_1}\le d \le x^B\big\},
\end{split}
\end{equation*}
where $D=D(c,B)$ is chosen as in Lemma~\ref{(carm)lem:cotton}.
\end{lemma}

\begin{proof}[Sketch of Proof]
We follow the proof and use the notation of \cite[Theorem~3.1]{AGP}.
In view of Lemma~\ref{(carm)lem:cotton} we can
replace the lower bound \cite[(3.2)]{AGP} with the bound
$$
\pi_c(dx^{1-B};d,1)
\ge\frac{1}{2}\,\frac{(dx^{1-B})^\gamma}{\varphi(d)\log x}
\qquad(d\mid L',~1\le d\le x^B).
$$
Also, since $dq\le(dx^{1-B})^A$ for any natural numbers
$d\le x^B$ and $q\le x^{(A-B)/2}$, Theorem~\ref{(carm)thm:heavymetal}
enables us to replace the upper bound that occurs after \cite[(3.2)]{AGP} 
with the bound
$$
\pi_c(dx^{1-B};dq,1)
\le\frac{4C}{q(1-A)}\frac{(dx^{1-B})^\gamma}{\varphi(d)\log x}
\qquad(1\le d\le x^B)
$$
for every prime $q$ dividing $L'$. Taking into account~\eqref{(carm)eq:acdc}, 
we see that there are at least
$$
\frac{(x^{1-B})^\gamma}{4\log x}
\sum_{\substack{1\le d\le x^B\\d\,\mid\,L'}}
\frac{d^\gamma}{\varphi(d)}\ge
\frac{(x^{1-B})^\gamma}{4\log x}\,x^{B_1(\gamma-1)}
\,\#\big\{d\mid L':x^{B_1}\le d \le x^B\big\}
$$
pairs $(p,d)$ where $p\le dx^{1-B}$ is a prime in $\cP^{(c)}$,
$p\equiv 1\bmod L$, $(p-1)/d$ is coprime to $L$, $d\mid L'$,
and $x^{B_1}\le d\le x^B$.  Hence, there is an integer $k\le x^{1-B}$
with $\gcd(k,L)=1$ such that $k$ has at least
$$
\frac{(x^{1-B+B_1})^{\gamma-1}}{4\log x}
\,\#\big\{d\mid L':x^{B_1}\le d \le x^B\big\}
$$
representations as $(p-1)/d$ with a pair $(p,d)$ as above.
Since we can replace \cite[(3.1)]{AGP} with the
lower bound
$$
\#\big\{d\mid L':x^{B_1}\le d \le x^B\big\}
\ge 2^{-D}\,\#\big\{d\mid L:x^{B_1}\le d \le x^B\big\},
$$
the proof is complete.
\end{proof}

Let $\pi(x)$ be the number of primes $p\le x$, and let $\pi(x,y)$
be the number of those for which $p-1$ is free of prime factors
exceeding $y$.  As in \cite{AGP}, we denote by $\cE$ the set of numbers
$E$ in the range $0<E<1$ for which 
$$
\pi(x,x^{1-E})\ge x^{1+o(1)}\qquad(x\to\infty),
$$
where the function implied by $o(1)$ depends only on $E$.
With only some slight modifications to the
proof of \cite[Theorem~4.1]{AGP}, using Lemma~\ref{(carm)lem:cotton2}
in place of\break \cite[Theorem~3.1]{AGP}, we have:

\begin{lemma}
\label{(carm)lem:tmp-quant}
Fix $c\in\(1,\tfrac{57}{56}\)$, and let $B,B_1$ be positive
real numbers such that $B_1<B<-\tfrac{17}{39}+\tfrac{6\gamma}{13}$.
For any $E\in\cE$ there is a number $x_4$ depending on $c,B,B_1,E$ and $\eps$, 
such that for any $x\ge x_4$ there are at least
$x^{EB+(1-B+B_1)(\gamma-1)-\eps}$ Carmichael numbers
up to $x$ composed solely of primes from $\sP^{(c)}$.
\end{lemma}

\begin{remark}
It may seem more natural to state this result for any
$c\in\(1,\tfrac{18}{17}\)$ in view of our earlier results; however, it
can be seen that the exponent $EB+(1-B+B_1)(\gamma-1)-\eps$ is never
positive when $c\ge\tfrac{57}{56}$, so the result is vacuous in that
case. This point is discussed further below.
\end{remark}

\begin{proof}[Sketch of Proof]
Following the proof and notation of \cite[Theorem~4.1]{AGP},
the condition~\eqref{(carm)eq:acdc} is easily verified, so we can construct
a set $\cP$ of primes in $\sP^{(c)}$ with $p\le x$ with $p=dk+1$
for some divisor $d$ of $L$, which satisfies the lower bound
$$
\#\cP\ge\frac{2^{-D-2}(x^{1-B+B_1})^{\gamma-1}}{\log x}
\,\#\big\{d\mid L:x^{B_1}\le d \le x^B\big\}
$$
by Lemma~\ref{(carm)lem:cotton2} (compare to \cite[(4.5)]{AGP}).
To complete the argument, we simply observe that the lower bound
for $\#\big\{d\mid L:1\le d \le x^B\big\}$ given on \cite[page~718]{AGP}
is also a lower bound for $\#\big\{d\mid L:x^{B_1}\le d \le x^B\big\}$
if $x$ is large enough, since the product of any
$$
u=\fl{\frac{\log x^B}{\log y^\theta}}=\fl{\frac{B\log x}{\theta\log y}}
$$
primes $q\in(y^\theta/\log y,y^\theta]$ is a divisor $d$ of $L$
of size $x^{B+o(1)}\le d\le x^B$ as $x\to\infty$.
\end{proof}

Taking $B$ and $B_1$ arbitrarily close to $-\tfrac{17}{39}+\tfrac{6\gamma}{13}$,
and noting that $\cE$ is an open set by \cite[Proposition~5.1]{AGP},
Lemma~\ref{(carm)lem:tmp-quant} implies that there are infinitely many
Carmichael numbers composed of primes from $\sP^{(c)}$ provided
that
\begin{equation}
\label{(carm)eq:firefly}
E\bigl(-\tfrac{17}{39}+\tfrac{6\gamma}{13}\bigr)+\gamma-1>0.
\end{equation}
Since $E<1$, this inequality cannot hold if $\gamma\ge\tfrac{56}{57}$.
Moreover, we do not know that $E$ can be taken arbitrarily close to one,
i.e., that $\cE=(0,1)$.  At present, it is known unconditionally
that $0.7039\in\cE$ (see Baker and Harman~\cite{BakHar}), and taking $E=0.7039$
in~\eqref{(carm)eq:firefly} leads to the statement of Theorem~\ref{(intro)thm:8}.

\end{document}